\documentclass[hidelinks, 12pt]{amsart}

\usepackage{amsmath}
\usepackage[psamsfonts]{amssymb}
\usepackage{amsfonts}
\usepackage{algorithm2e}
\usepackage{graphicx}
\usepackage{verbatim}
\usepackage{dsfont}
\usepackage{tikz-cd}
\usepackage{mathrsfs}
\usepackage{color}
\usepackage[english]{babel}
\usepackage{fontenc}
\usepackage{indentfirst}
\usepackage{tikz}
\usetikzlibrary{matrix,arrows,decorations.pathmorphing}
\usepackage{algorithm2e}
\usepackage[all]{xy}
\usepackage[T1]{fontenc}
\usepackage{hyperref}
\usepackage{mathtools}
\usepackage{multicol}

\newtheorem{theorem}{Theorem}[section]
\newtheorem{prop}[theorem]{Proposition}
\newtheorem{lemma}[theorem]{Lemma}
\newtheorem{cor}[theorem]{Corollary}

\newtheorem{ex}[theorem]{Example}
\newtheorem{dfn}[theorem]{Definition}
\newtheorem{remark}[theorem]{Remark}

\newcommand{\abs}[1]{\left\lvert#1\right\rvert}
\newcommand{\norm}[1]{\left\lVert#1\right\rVert}
\newcommand{\op}[1]{\operatorname{#1}}
\renewcommand{\hat}[1]{\widehat{#1}}
\numberwithin{equation}{section}
\DeclareMathOperator{\im}{im}

\DeclareMathOperator{\R}{\mathbb{R}}

\def\ep{\epsilon}

\def\R{\mathbb{R}}
\def\Z{\mathbb{Z}}
\def\N{\mathbb{N}}

\def\cal R{\mathcal R}

\def\k{{\bf k}}
\def\o{\it{o}}
\def\V{\mathbb V}

\DeclarePairedDelimiter\ceil{\lceil}{\rceil}
\DeclarePairedDelimiter\floor{\lfloor}{\rfloor}

\author{Danijel Djordjevi\'c}
\email{danijel$\_$djordjevic@matf.bg.ac.rs}
\address{Faculty of mathematics, University of Belgrade, Studentski trg 16, 11158 Belgrade,
Serbia}

\author{Igor Uljarevi\'c}
\email{igoru@matf.bg.ac.rs}
\address{Faculty of mathematics, University of Belgrade, Studentski trg 16, 11158 Belgrade,
Serbia}

\author{Jun Zhang}
\email{jzhang4518@ustc.edu.cn}
\address{The Institute of Geometry and Physics, University of Science and Technology of China, 96 Jinzhai Road, Hefei Anhui, 230026, China}
 
\title{Quantitative characterization in contact Hamiltonian dynamics - I}

\begin{document}

\maketitle

\begin{abstract} Based on the contact Hamiltonian Floer theory established in \cite{MU} that applies to any (admissible) contact Hamiltonian system $(M, \xi = \ker \alpha, h)$, where $h$ is a contact Hamiltonian function on a Liouville fillable contact manifold $(M, \xi = \ker \alpha)$, we associate a persistence module to $(M, \xi, h)$, called a gapped module, that is parametrized only by a partially ordered set. It enables us to define various numerical Floer-theoretic invariants. In this paper, we focus on the contact spectral invariants and their applications. Several key properties are proved, which include stability with respect to the Shelukhin-Hofer norm in \cite{Egor-cont} and a triangle inequality of the contact spectral invariant. In particular, our stability property does {\it not} involve any conformal factors; our triangle inequality is derived from a novel analysis on pair-of-pants in the contact Hamiltonian Floer homology. While this paper was nearing completion, the authors were made aware of upcoming work \cite{Cant_shelukhin}, where a similar persistence module for a contact Hamiltonian dynamics was constructed. \end{abstract}


\section{Introduction}

\subsection{Motivation from contact Hamiltonian Floer homology} Similarly to the Hamiltonian dynamics in symplectic geometry, Hamiltonian dynamics also makes sense on a contact manifold $(M, \xi = \ker \alpha)$, with a fixed contact 1-form $\alpha$. Explicitly, given a Hamiltonian function $h: [0,1] \times M \to \R$, by solving the following equations,
\begin{equation} \label{ham-equation}
\alpha(X_h) = h \,\,\,\,\mbox{and}\,\,\,\, \iota_{X_h}d\alpha + dh = dh(R_{\alpha}) 
\end{equation}
one obtains a vector field $X_h$ whose flow, denoted by $\phi_h^t$, preserves the contact structure $\xi$. Along with the spirit of the (symplectic) Hamiltonian Floer homology constructed from the closed orbits of a Hamiltonian flow, one attempts to construct a  contact Hamiltonian Floer homology, still denoted by ${\rm HF}_*(M, h; \k)$ where $\k$ is a fixed field, with generators being the closed orbits of the contact Hamiltonian flow $\{\phi_h^1\}_{t \in [0,1]}$. 

However, this is difficult (at least non-trivial) to be realized, partially because the real 2-dimensional cylinders connecting two closed orbits in the standard definition of the Floer boundary map on the corresponding chain complex 
\[ \partial_*: {\rm CF}_*(M, h; \k) \to {\rm CF}_{*-1}(M, h; \k) \]
sit inside an odd-dimensional ambient space $M$, where the traditional approach such as the pseudo-holomorphic curve theory does not apply directly. 

When $(M, \xi = \ker \alpha)$ is symplectic fillable, \cite{MU} establishes a version of contact Hamiltonian Floer homology ${\rm HF}_*(M, h; \k)$ or briefly ${\rm HF}_*(h)$ if there is no need to emphasize the contact manifold $M$ or the ground field $\k$. The construction is inspired by the classical Viterbo's construction of symplectic homology in \cite{Viterbo-spec}.   
Whenever $h$ satisfies certain admissible condition (see Section \ref{sec-construction}), it is defined via a (symplectic) Hamiltonian Floer homology as follows, 
\begin{equation} \label{MU-def}
{\rm HF}_*(M, h; \k) : = {\rm HF}_*(\widehat{W}, H; \k) 
\end{equation}
where $\widehat{W}$ is the completion of a symplectic filling $W$ of $M$, that is, $\widehat{W} \backslash W = [1, \infty)(r) \times M$,  and $H: [0,1] \times \widehat{W} \to \R$ satisfies 
\begin{equation} \label{linear-H}
H = r \cdot h +C \,\,\,\,\mbox{for some constant $C$}
\end{equation}
when the ``radius'' $r$ is sufficiently large. In other words, ${\rm HF}_*(M, h; \k)$ can be regarded as an intermediate step in Viterbo's construction where the constant slope is replaced by a function $h$ (hence, sometimes $h$ is called the {\bf slope of $H$}). 

Of course, the well-definedness of ${\rm HF}_*(M, h; \k)$ needs a justification, and it turns out that when $H, G: [0,1] \times \hat{W} \to \R$ have the same slope $h$ as in (\ref{linear-H}) (but no need to start from the same radius $r$), then we have 
\begin{equation} \label{H-iso}
{\rm HF}_*(\widehat{W}, H; \k) \simeq {\rm HF}_*(\widehat{W}, G; \k)
\end{equation}
as $\k$-modules. As a Hamiltonian Floer homology on a non-compact symplectic manifold, the homology ${\rm HF}_*(M, h; \k)$ does {\it not} share all the properties as for a closed symplectic manifold. In particular, \cite{UZ22} investigates in detail the behavior of the Floer continuation map for two different (contact) Hamiltonian $h, g: [0,1] \times M \to \R$. Surprisingly, whether the Floer continuation map induces an isomorphism between ${\rm HF}_*(M, -; \k)$ is related to certain rigidity of positive loops in ${\rm Cont}_0(M, \xi)$. 

In this paper, we continue to investigate whether other structures from (symplectic) Hamiltonian Floer homology can be adapted to or in which way they are shared by the contact Hamiltonian Floer homology ${\rm HF}_*(M, h; \k)$. To simplify our discussion, we will consider Liouville fillable $(M, \xi = \ker \alpha)$, where the filling is exact by definition. Then ${\rm HF}_*(M, h; \k)$ is well-defined and finite dimensional over $\k$. For more sophisticated filling, the Novikov field could be involved and we will explore this elsewhere. 

As an infinitely-dimensional Morse theory with the Morse function defined by the action functional $\mathcal A_H$, the (symplectic) Hamiltonian Floer homology can upgrade to a persistence $\k$-module (see the beginning of Section \ref{sec-gap-mod} for a brief introduction of its background and \cite{CZ05,CC-SGGO-prox-09,CdeSGO16,PRSZ20} for more details). This is a relatively simple algebraic structure that enriches the classical homological theory in such a way that homologically {\bf invisible} generators, as well as their homological-killing relation with the respect to the Floer boundary operator, can be studied systematically. As an application, it serves as a uniform perspective to extract lots of homological invariants from Hamiltonian Floer homology, including spectral invariant as well as spectral norm \cite{Oh-spec,Viterbo-spec,Schwarz}, boundary depths \cite{Usher-bd-1,Usher-bd-2,UZ16}, etc. They have made fundamental contributions to the quantitative studies in the modern symplectic geometry, starting from Polterovich-Shelukhin's pioneer work \cite{PS-pers-16}.  However, for ${\rm HF}_*(M, h; \k)$, its definition, justified by the isomorphism in (\ref{H-iso}), possesses a certain amount of ambiguity in the choice of the Hamiltonian function $H: [0,1] \times W \to \R$. Therefore, the standard way to filter the Hamiltonian Floer homology via the functional $\mathcal A_H$ fails for ${\rm HF}_*(M, h; \k)$. In this paper, we propose a novel approach to obtain a persistence $\k$-module structure for ${\rm HF}_*(M, h; \k)$ and read off several useful invariants from it. 

\begin{remark} A recent series of papers from Oh \cite{Oh-21,Oh-21-b,Oh-22, Oh-22-b, Oh-22-c} established a new approach to define a Floer homological theory on a contact manifold {\rm without} passing to any form of its symplectization. Compared with our approach, it is more closer to the standard (symplectic) Hamiltonian Floer type theory where an appropriate action functional is constructed. This leads to the construction of some numerical invariant, e.g. contact spectral invariant, in certain cases {\rm (}see \cite{Oh-Yu-a,Oh-Yu}{\rm )}. \end{remark}

\subsection{Construction of a variant to persistence module} \label{sec-construction} The starting point comes from the well-known formula of the composition of two contact Hamiltonian isotopies, that is, if $h, g: [0,1] \times M \to \R$ generate $\phi_h^t$ and $\phi_g^t$ respectively, then 
\begin{equation} \label{cocycle}
 g\#h : = g_t + \left(e^{\kappa(\phi_g^t)} h_t\right) \circ \phi_g^{-t}  \,\,\,\,\mbox{generates} \,\,\,\, \phi_g^t \circ \phi_h^t
 \end{equation}
where $\kappa(\phi_g^t)$ denotes the conformal factor of $\phi_g^t$. In particular, consider $g_t(x) \equiv \eta \in \R$, a constant function, then $\phi_g^t = \phi_{R}^{\eta t}$ (a reparametrization of the Reeb flow) and $\kappa(\phi_g^t) = 0$. Hence, 
\begin{equation} \label{comp}
(\eta \# h)_t = \eta + h_t \circ \phi_{R}^{-\eta t},
\end{equation}
which generates the contact Hamiltonian isotopy $\phi_R^{\eta t} \circ \phi_h^t$. Once the admissible condition can be verified, one can define the contact Hamiltonian Floer homology ${\rm HF}_*(M, \eta \# h; \k)$. 

To obtain a persistence $\k$-module structure, we will vary $\eta  \in \R$. By the monotone continuation property of ${\rm HF}_*$ in \cite{U-Viterbo}, once a pair of constant functions $\eta \leq \eta' \in \R$ satisfy the following condition: either $\eta = \eta'$ or 
\begin{equation} \label{gap-cond}
\eta' - \eta \geq \int_0^1 \max_M h_t - \min_M h_t \, dt \,\,=: {\rm osc}(h)
\end{equation}
one obtains a well-defined $\k$-linear map $\iota: {\rm HF}_*(M, \eta \# h; \k) \to {\rm HF}_*(M, \eta' \# h; \k)$ since $(\eta \#h)_t \leq (\eta' \#h)_t$ pointwise. Indeed, when $\eta = \eta'$, the induced map $\iota = \mathds{1}$; otherwise by (\ref{comp}) and (\ref{gap-cond}), for each $t \in [0,1]$, 
\begin{align*}
(\eta' \# h)_t - (\eta \# h)_t & = \left(\eta' + h_t \circ \phi_{R}^{-\eta' t}\right) -  \left(\eta + h_t \circ \phi_{R}^{-\eta t}\right)\\
& = (\eta' - \eta) - \left(h_t \circ \phi_R^{-\eta t} - h_t \circ \phi_R^{-\eta' t}\right) \\
& \geq \left(\max_M h_t - \min_M h_t\right)  - \left(h_t \circ \phi_R^{-\eta t} - h_t \circ \phi_R^{-\eta' t}\right) \geq 0. 
\end{align*}
Note that it seems difficult to improve the condition (\ref{gap-cond}) since the value $h_t \circ \phi_{R}^{-\eta t}$ is hard to control in general. A possible way comes from the observation that we are comparing gap $\eta' - \eta$ with the oscillation of $h_t$ on each Reeb flow trajectory, which could improve ${\rm osc}(h)$. For instance, if $h$ generates a strict contact isotopy (i.e., preserving the contact 1-form), then $h_t$ is constant along each Reeb flow trajectory and the requested gap $\eta' - \eta$ is simply $0$. 

\medskip

To proceed, we need to specify the admissible condition for a contact Hamiltonian function in this paper. 


\medskip

\noindent {\it Admissible condition}. Given a Liouville fillable contact manifold $(M, \xi = \ker\alpha)$ with a fixed filling $W$, a contact Hamiltonian $h: [0,1] \times M \to \R$ is called {\it admissible} if the lift $H(r,x) = e^rh(x)$ to the convex end of the completion $\widehat{W}$, viewed as a part of the symplectization $SM: = (\R \times M, d(e^r \alpha))$, generates a Hamiltonian diffeomorphism $\Phi = \Phi_H^1$ defined by 
\begin{equation} \label{lift-Ham}
\Phi(r,x) = (r - \kappa(\phi_h^1)(x), \phi_h^1(x))
\end{equation}
that does not have any fixed points in this convex end. Note that this is weaker than the original definition of being admissible in \cite{MU}, which requires that the time-1 map $\phi_h^1$ does not have any fixed points in $M$. 

More concretely, for the contact Hamiltonian $\eta \# h$ in discussion above, it is admissible if in the convex end there do not exist points $(r,x)$ satisfying 
\begin{equation} \label{time-1-translated}
r - \kappa(\phi_h^1)(x) = r \,\,\,\,\mbox{and}\,\,\,\, \phi_R^{\eta}(\phi_h^1(x)) = x. 
\end{equation}
Then $\kappa(\phi_h^1)(x)=0$ and $\phi_h^1(x) = \phi_R^{-\eta}(x)$, which precisely corresponds to the {\bf translated points of $\phi_h^1$ with time shift $-\eta$}, according to Sandon's definition in \cite{San-tp}. This leads to the following crucial observation, which plays an important role in several arguments later. 

\begin{lemma} \label{lemma-spec} Given a Liouville fillable contact manifold $(M, \xi = \ker \alpha)$ and a contact Hamiltonian $h: [0,1] \times M \to \R$, the set of $\eta \in \R$ where the contact Hamiltonian Floer homology ${\rm HF}_*(M, \eta\#h; \k)$ is {\rm not} admissible is discrete in $\R$. \end{lemma}

\begin{proof} By Definition 2.5 in \cite{AM}, one defines the Rabinowitz-Floer action functional $\mathcal A_h: \Lambda(SM) \times \R \to \R$, where $\Lambda(SM)$ is the loop space of the symplecitzation of $M$, such that the critical points $(\gamma, \eta) = ((x(t), r(t))_{t \in [0,1]}, \eta)$ of $\mathcal A_h$ precisely correspond to the translated points $p$ of $\phi_h^1$ with time shift $-\eta$ (caution: the sign!). Moreover,
\begin{equation} \label{critical point data}
x\left(1/2\right) \longleftrightarrow p \,\,\,\,\mbox{and}\,\,\,\, \mathcal A_h((\gamma, \eta)) = \eta.
\end{equation}
For a complete argument of this fact, see Lemma 2.7 in \cite{AM}. By the discussion above and (\ref{critical point data}), the requested set $\eta \in \R$ where the contact Hamiltonian Floer homology ${\rm HF}_*(M, \eta\#h; \k)$ is not well-defined is equal to the following set of the negative spectrum of functional $\mathcal A_h$, that is,
\begin{equation} \label{dfn-RFspec}
{\rm spec}(\mathcal A_h) = \left\{-\eta \in \R \,\bigg| \, \begin{array}{ll} \mbox{$\eta$ is the time shift of some} \\ \mbox{translated point of $\phi_h^1$} \end{array}\right\}. 
\end{equation}
Then, by Sard's theorem (applied to $\mathcal A_h$), the set ${\rm spec}(\mathcal A_h)$ is a nowhere dense and closed subset of $\R$, in particular, discrete. \end{proof}

In what follows, to simplify the notations, let us denote the discrete set ${\rm spec}(\mathcal A_h)$ promised in Lemma \ref{lemma-spec} by $\mathcal S_h$. For any strictly increasing sequence $\mathfrak a =  \{\eta_i\}_{i \in \Z}$ satisfying, for any $i \in \Z$, 
\begin{equation} \label{gap-seq}
 \eta_{i+1} - \eta_i \geq {\rm osc}(h), \,\,\,\,\mbox{and}\,\,\,\,\eta_i \in \R \backslash \mathcal S_h, 
 \end{equation}
the associated contact Hamiltonian Floer homologies form a $\Z$-index persistence $\k$-modules, 
\begin{equation} \label{gapped per mod}
\mathbb V_h(\mathfrak a): = \left(\left\{{\rm HF}_*(\eta_i \#h)\right\}_{i \in \Z}, \{\iota_{i, i+1}: {\rm HF}_*(\eta_i \#h) \to {\rm HF}_*(\eta_{i+1} \#h)\}_{i \in \Z}\right).
\end{equation}
Due to the following double identifications, 
\[ \Z \longleftrightarrow \{\eta_i\}_{i \in \Z} \longleftrightarrow {\rm HF}_*(\eta_i \#h),\] 
we are able to view $\mathbb V_h(\mathfrak a)$ as a persistence $\k$-module over $\Z$, instead of over the subset $\{\eta_i\}_{i \in \Z}$ in $\R$.

The resulting persistence $\k$-module $\mathbb V_h(\mathfrak a)$ is different from those constructed via the filtration given by symplectic action functions (which mostly are $\R$-indexed). One the one hand, it is $\Z$-indexed or discretely-indexed in general; on the other hand, it depends on the choice of a sequence $\mathfrak a$ satisfying (\ref{gap-seq}). Within all possible sequences $\mathfrak a$, one defines a partial order $\mathfrak a = \{\eta_i\}_{i \in \Z} \preceq \mathfrak b = \{\tau_i\}_{i \in \Z}$ if for any $\tau_i$, there exists $j(i) \in \Z$ such that $\eta_{j(i)} \leq \tau_i \leq \eta_{j(i)+1}$. Due to the first condition in (\ref{gap-seq}), observe that there exists a family of ``minimal '' $\mathfrak{a}$ (in terms of $\preceq$) in the form of 
\begin{align*}
\mathfrak{a} & = \{\cdots, \eta_0, \eta_1,  \eta_2, \cdots \}\\
& = \{\cdots, \eta_0, \eta_0 + {\rm osc}(h) + o_1(1),  \eta_0 + 2{\rm osc}(h) + o_2(1), \cdots \} 
\end{align*}
It is almost an ${\rm osc}(h)$-arithmetic sequence, where the set $\{o_i(1)\}_{i \in \Z}$ represents small adjustments that may vary along $i$, in order to satisfy the second condition in (\ref{gap-seq}). For each such minimal sequence $\mathfrak a$, the corresponding persistence $\k$-module $\mathbb V_h(\mathfrak a)$ is an example of a so-called {\it almost optimal restriction of an ${\rm osc}(h)$-gapped $\k$-module} which will be discussed in detail in Section \ref{sec-gap-mod}. 

\medskip

In terms of a purely algebraic formulation, for $s, t \in \mathcal \R \backslash \mathcal S_h$, denote by 
\begin{equation} \label{partial-order}
\eta \leq_{{\rm osc}(h)} \lambda: \,\,\,\,\mbox{either $\eta =\lambda$ or $\eta \leq \lambda -{\rm osc}(h)$}.
\end{equation}
This relation defines a partial order on $\mathcal \mathcal \R \backslash \mathcal S_h$. Then, the following persistence module 
\[ \left(\left\{{\rm HF}_*(\eta \#h)\right\}_{\eta \in \R \backslash \mathcal S_h}, \{\iota_{\eta, \lambda}: {\rm HF}_*(\eta \#h) \to {\rm HF}_*(\lambda \#h)\}_{\eta \leq_{{\rm osc}(h)} \lambda \in \R \backslash \mathcal S_h}\right)\]
is, to our best knowledge, the first concrete example in symplectic and contact geometry with parametrization that is {\bf not necessarily totally ordered}. Interestingly, for such a persistence module, \cite[Theorem 1.1]{BC-B-decomposition} still guarantees a decomposition, but at present we are not clear about the indecomposable components in this decomposition (cf.~\cite[Theorem 1.2]{BC-B-decomposition}), which results an obstruction to describe the barcode directly. 

\begin{remark} We on purpose choose the contact Hamiltonian to be $\eta \#h$ instead of $h \# \eta$. First, these two functions are different in general, since by definition, 
\begin{equation} \label{wrong-order}
(h\# \eta)_t = h_t + e^{\kappa(\phi_h^t)} \eta 
\end{equation}
where in particular the conformal factor of $\phi_h^t$ appears. Second, the sequence of contact Hamiltonian Floer homologies $\{{\rm HF}_*(h \# \eta)\}_{\eta \in \R \backslash \mathcal S_h}$ form a standard persistence module in the sense that if $\eta \leq \eta' \in \R \backslash \mathcal S_h$, then there exists a well-defined Floer continuation map $\iota_{\eta, \eta'}: {\rm HF}_*(h \# \eta) \to {\rm HF}_*(h \#\eta')$. Third, when comparing the resulting persistence modules associated to different contact Hamiltonians $h$ and $g$, in particular to obtain a Floer continuation map 
\[ {\rm HF}_*(h \#\eta) \to {\rm HF}_*(g \# \eta'), \]
one needs $g \# \eta' \geq h \# \eta$ pointwise. Since it is not obvious how to isolate $\eta$ and $\eta'$ from the conformal factors $\kappa(\phi_h^t)$ and $\kappa(\phi_g^t)$, the two persistence modules may differ a lot, say, in terms of the interleaving distance. \end{remark}

\begin{remark} \label{c-DC} While this paper is close to being completed, we were informed by Dylan Cant that a persistence module from the contact Hamiltonian Floer homology ${\rm HF}_*(M, h; \k)$ was also constructed in his upcoming work \cite{Cant_shelukhin}, with a certain amount of similarity to our construction. More precisely, he also starts from the contact Hamiltonians in a similar form of (\ref{comp}). Different from our gapped module approach proposed in this section with more details in Section \ref{sec-gap-mod}, he is able to obtain a standard persistence module parametrized by $\R$ or a dense subset of $\R$, via twisted cylinders along a homotopy $\{h\#\eta_{s}\}_{s \in [0,1]}$ from $h\#\eta$ to $h\# \eta'$ for any two $\eta \leq \eta'$. For more details, see \cite[Section 2]{Cant_shelukhin}. This overlap leads to several similar applications as in our Section \ref{sec-mra}. \end{remark}

To end this section, let us emphasize the following point. The contact Hamiltonian dynamics on $(M, \xi = \ker \alpha)$ is sensitive to the contact 1-form $\alpha$ (that gives the same contact structure $\xi$). For instance, for different such contact 1-forms $\alpha$ and $\beta$, the corresponding Reeb flows $\phi_{R, \alpha}^t$ and $\phi_{R, \beta}^t$ may differ a lot by a simple calculation (or see Proposition 2.14 in \cite{Oh-21}). However, it is easy to see that for any $\eta \in \R$, we have 
\[ \max_M \left|(\eta \#_{\alpha} h)_t - (\eta \#_{\beta} h)_t \right|  \leq {\rm osc}(h) \]
for any $t \in [0,1]$, where $\#_{\alpha}$ denotes the sum in (\ref{comp}) under the contact 1-form $\alpha$, similarly to $\#_{\beta}$. This implies that ${\rm osc}(h)$-gapped modules $\mathbb V_{h, \alpha}$ and $\mathbb V_{h, \beta}$, with respect to $\alpha$ and $\beta$ respectively, are at most ${\rm osc}(h)$-interleaved (see Definition \ref{dfn-gap-inter}). 

\section{Main results and applications} \label{sec-mra} Let $(M, \xi= \ker \alpha)$ be a Liouville fillable contact manifold with a fixed filling $(W,\lambda)$. Before we present the main results of this paper, let us introduce some notations. For two Hamiltonian functions $h, g: [0,1] \times M \to \R$, similarly to the definition of ${\rm osc}(h)$ and ${\rm osc}(g)$ as in (\ref{gap-cond}), define 
\begin{equation} \label{dfn-osc}
{\rm osc}(h,g) = \int_0^1 \max_M h_t - \min_M g_t \,dt.
\end{equation}
Similarly, one defines ${\rm osc}(g,h)$. Observe that ${\rm osc}(h,g) + {\rm osc}(g,h) = {\rm osc}(h) + {\rm osc}(g) \geq 0$, so, at least one of ${\rm osc}(h,g)$ and ${\rm osc}(g,h)$ must be non-negative. Now, denote by 
\begin{equation} \label{dfn-m-hg}
m_{h,g} : = 2\max\{{\rm osc}(g,h), {\rm osc}(h, g)\} (\geq 0).
\end{equation}

Here is the first main result of the paper. 

\begin{theorem} \label{intro-thm-1} For any admissible contact Hamiltonian $h: [0,1] \times M \to \R$, there exists an ${\rm osc}(h)$-gapped module (see Definition \ref{dfn-gap-mod}) associated to $h$, denoted by $\mathbb V_h$. Moreover, two such gapped modules $\mathbb V_h$ and $\mathbb V_g$ are $m_{h,g}$-interleaved (see Definition \ref{dfn-gap-inter}) if $h \neq g$ and $0$-interleaved if $h = g$.\end{theorem}

The algebraic preparation for the proof of Theorem \ref{intro-thm-1} occupies the entire Section \ref{sec-gap-mod}, which builds up the general theory of gapped modules. The detailed proof of Theorem \ref{intro-thm-1} will be given in Section \ref{sec-proof-main}. Similarly to the standard persistence module, one can read off numerical data from $\mathbb V_h$. In this paper, we will be mainly interested to the contact spectral invariant $c(a, \mathbb V_h)$, as the spectral invariant of the ${\rm osc}(h)$-gapped module $\V_h$ (see Definition \ref{dfn-si-gap}). Here, $a$ is a class belonging to the limit 
\begin{equation} \label{final-slice}
 (\V_h)_{\infty} = {\rm SH}_*(W; \k).
 \end{equation}
Note that (\ref{final-slice}) holds due to the definition of ${\rm SH}_*$ in terms of the direct limit (see \cite{Viterbo-spec}) and this limit is independent of the choice of contact Hamiltonian function $h$. 

\medskip

The following several main results list key properties of $c(a, \mathbb V_h)$ which justify the name  - contact spectral invariant. 
 
\begin{theorem} [Stability] \label{prop-stability} Let $(M, \xi)$ be a contact manifold with a Liouville filling $(W, \lambda)$. Then, for any admissible contact Hamiltonians $h, g: [0,1] \times M \to \R$ and a non-zero class $a \in {\rm SH}_*(W)$, we have 
\[ |c(a, \mathbb V_h) - c(a, \mathbb V_g)| \leq 2\max\{{\rm osc}(g,h), {\rm osc}(h, g)\}. \] 
In particular, when $h = g$, we have the obvious equality $c(a,h) = c(a,g)$. 
\end{theorem}

In fact, Theorem \ref{prop-stability} is an immediate corollary of Theorem \ref{intro-thm-1}, plus an algebraic stability - Proposition \ref{prop-alg-stability}. Keeping track of its proof, one obtains the following monotonicity property. 

\begin{cor} [Monotonicity] \label{cor-monotone} Let $(M, \xi)$ be a contact manifold with a Liouville filling $(W, \lambda)$. Then, for any contact Hamiltonians $h, g: [0,1] \times M \to \R$ with $h_t \leq g_t$ pointwise, we have $c(a, \mathbb V_h) \leq  c(a, \mathbb V_g)$ for any a non-zero class $a \in {\rm SH}_*(W)$, \end{cor}

On the other hand, letting $g = 0$ in Theorem \ref{prop-stability}, we have the following result. 

\begin{cor} [Stability with respect to the Shelukhin-Hofer norm] \label{cor-Egor} Let $(M, \xi)$ be a contact manifold with a Liouville filling $(W, \lambda)$. Then, for any admissible contact Hamiltonians $h: [0,1] \times M \to \R$ and a non-zero class $a \in {\rm SH}_*(W)$, we have 
\begin{equation} \label{est-Egor}
|c(a, \mathbb V_h) - c(a, \mathbb V_0)| \leq 2 \int_0^1 \max_M |h_t| \,dt.
\end{equation}
\end{cor}
\begin{proof} By Theorem \ref{prop-stability}, we have
\begin{align*}
|c(a, \mathbb V_h) - c(a, \mathbb V_0)| & \leq 2\max\left\{\int_0^1 -\min_M h_t\, dt, \int_0^1 \max_M h_t \,dt\right\} \\
& \leq 2 \int_0^1 \max_M |h_t| \,dt 
\end{align*}
since $\max\{-\min_M h_t, \max_M h_t\} \leq \max_M |h_t|$ for each $t \in [0,1]$. 
\end{proof}

The estimation (\ref{est-Egor}) from Corollary \ref{cor-Egor} does {\rm not} imply that when $h \to 0$, we have $c(a, \mathbb V_h) \to 0$ since the term $c(a, \mathbb V_0)$ may not be zero (see the computational Example in Section \ref{sec-ex}). This is an essential difference from the spectral invariant in symplectic geometry. Moreover, the upper bound in (\ref{est-Egor}) is considered by Shelukhin in \cite{Egor-cont} to formulate a Hofer-like norm on the contactomorphism group. To our best knowledge, $c(a, \mathbb V_h)$ is the first numerical invariant in contact geometry that is stable with respect to Shelukhin's Hofer-like norm. 

\begin{theorem}[Shift] \label{prop-shift} Let $(M, \xi)$ be a contact manifold with a Liouville filling $(W, \lambda)$. Then, for any contact Hamiltonians $h: [0,1] \times M \to \R$ and a class $a \in {\rm SH}_*(W)$, we have $$c(a, \mathbb V_{s \#h}) = c(a, \mathbb V_h) + s$$ for any $s \in \R$. Here, recall that $s\#h = h \circ \phi_R^{-st} + s$. \end{theorem}

Note that Theorem \ref{prop-shift} enables us to easily obtain more non-trivial computational results for the contact spectral invariant $c(-, \mathbb V_s)$ where $s \in \R$ is a constant function, starting from computing $c(-, \mathbb V_0)$. Different from the Hamiltonian dynamics, shifting by a contact $s$ gives rise to the Reeb dynamics. 

\medskip

Recall that $\mathcal S_h (={\rm spec}(\mathcal A_h))$ in (\ref{dfn-RFspec}) denotes the set of the (negative) time shifts of the translated points, which appears in the proof of Lemma \ref{lemma-spec}. 

\begin{theorem} [Spectrality] \label{prop-spectrality} Let $(M, \xi)$ be a contact manifold with a Liouville filling $(W, \lambda)$. Then, for any contact Hamiltonians $h: [0,1] \times M \to \R$ and a non-zero class $a \in {\rm SH}_*(W)$, we have $c(a,\mathbb V_h) \in \mathcal S_h$. \end{theorem}

The proof of Theorem \ref{prop-spectrality} could have been derived directly from Usher's spectrality theorem in \cite{Ush-spec}, however, since our Floer theory admits a non-standard filtration, for completeness, we give a detailed proof in Section \ref{sec-proof-main}. Here is an interesting application of Theorem \ref{prop-spectrality}. 

\begin{cor} [Existence of a translated point] \label{cor-exist-tp} Let $(M, \xi)$ be a contact manifold with a Liouville filling $(W, \lambda)$ with ${\rm SH}_*(W) \neq 0$. Then, for any non-zero contact Hamiltonians $h: [0,1] \times M \to \R$, the contactomorphism $\phi_h^1$ admits a translated point. \end{cor}

\begin{proof} If ${\rm SH}_*(W) \neq 0$, then for the unit $1$ of ${\rm SH}_*(W)$ and any $h: [0,1] \times M \to \R$, it has $c(1, \mathbb V_h) \neq +\infty$ by definition. In particular, by definition for any gapped sequence $\mathfrak a$, we have $c(1, \mathbb V_h(\mathfrak a)) \neq +\infty$. In terms of the standard barcode of $\mathbb V_h(\mathfrak a)$, there exists an infinite length interval with a finite left endpoint. Then Proposition \ref{prop-spectrality} yields that this left endpoint corresponds to the spectrum of a translated point of $\phi_h^1$, which implies the existence of a translated point of $\phi_h^1$. \end{proof}

\begin{remark} By the recent work \cite{Cant-no-tp}, it is possible that when ${\rm SH}_*(W) = 0$, there exists no translated points at all (so the contact Hamiltonian Floer homology as above does not provide any information, even the homological invisible ones). This reveals a peculiar phenomenon in contact Hamiltonian dynamics, where the (symplectic) Hamiltonian Floer homology never vanishes. \end{remark}

To state the next result, let us recall the Poincar\'e duality on symplectic (co)homology. In terms of the notation from \cite{CO}, we have the following duality, 
\begin{equation} \label{pd-sh}
{\rm PD}: {\rm SH}_*(W) \simeq {\rm SH}^{-*}(W, \partial W). 
\end{equation}
By definition, ${\rm SH}^{-*}(W, \partial W)$ is defined as the inverse limit of Hamiltonian Floer {\it homologies} ${\rm HF}_*(H)$ on the completion $\widehat{W}$, where $H(r,x) = ra + C$ on the cylindrical end of the completion $\widehat{W}$ and a {\it negative} constant function $a <0$. Therefore, we have 
\[ \varprojlim_{\eta} {\rm HF}_*(\eta \#h) = {\rm SH}^{-*}(W, \partial W) \]
where the inverse limit is taken over $\eta \in \R$ when $\eta \to -\infty$ (cf.~(\ref{final-slice})). On the level of chain complexes that define the symplectic (co)homologies, the duality in (\ref{pd-sh}) is induced by the following isomorphism,  
\begin{equation} \label{chain-iso}
{\rm CF}_*(H, J) \simeq {\rm CF}^{-*}(\bar{H}, \bar{J})
\end{equation}
where ${\rm CF}^{-*}(\bar{H}, \bar{J})$ can be algebraically identified with ${\rm Hom}({\rm CF}_{-*}(\bar{H}, \bar{J}); \k)$. Here, $\bar{H}(t,x) := -H(-t, x)$ and $\bar{J}_t := J_{-t}$. More explicitly, the isomorphism in (\ref{chain-iso}) is the identification of close Hamiltonian orbits by reversing the time. A subtle point is that, following the convention in Section 3.1 in \cite{CO}, the algebraic duality does {\it not} change the sign of the filtration since on the cochain complexes one considers super-level sets instead of sublevel sets (as a comparison, following the convention in Section 6.2 in \cite{UZ16}, the algebraic duality changes the sign of the filtration). All in all, here only the Poincar\'e duality in (\ref{chain-iso}) changes the sign of the filtration. 

\begin{theorem}[Duality]\label{prop-duality}  Let $(M, \xi)$ be a contact manifold with a Liouville filling $(W, \lambda)$. Then, for any contact Hamiltonians $h: [0,1] \times M \to \R$ and a class $a \in {\rm SH}_*(W)$, we have 
\[ c(a, \mathbb V_h) = - c({\rm PD}(a), \mathbb V_{\bar{h}}) \]
where ${\rm PD}$ is the Poincar\'e duality in (\ref{pd-sh}) and $\bar{h}(t,x) = -h(-t, x)$. 
\end{theorem}

The proof of Theorem \ref{prop-duality} will be given in Section \ref{sec-proof-main}. Imitating the (symplectic) spectral norm in symplectic geometry, one can define {\bf contact spectral norm} by 
\begin{equation} \label{dfn-sn}
\gamma_{\rm cont} (h) : = c({\rm PD}(1), \mathbb V_h) - c(1, \mathbb V_h)
\end{equation}
where as above $1$ denotes the unit of ${\rm SH}_*(W)$. Further properties and applications of $\gamma_{\rm cont}$ will be explored in the forthcoming work \cite{DUZ-part2}. 

\medskip

The following property is another main results in this paper - the triangle inequality for contact spectral invariants $c(\theta, \mathbb V_h)$. Before stating the inequality, we introduce the following notation:
\begin{enumerate}
    \item Denote by $h\hat{\#}g$ the concatenation of the contact Hamiltonians $h$ and $g$, i.e. $h\hat{\#}g$ is the contact Hamiltonian given by
    \[ \left(h\hat{\#}g\right)_t:= \left\{ \begin{matrix} h_{2t} & \text{for } t\in\left[0,\frac{1}{2}\right]\\ g_{2t-1}&\text{for } t\in\left[\frac{1}{2}, 1\right].\end{matrix} \right.\]
    \item Denote $\displaystyle\overline{\op{osc}}\: h:=\max_t\left(\max_x h_t(x)-\min_{x}h_t(x)\right)$.
\end{enumerate}

\begin{theorem}\label{prop:triangle}
    Let $(M, \xi)$ be a contact manifold with a Liouville filling $(W, \lambda)$ and $h_t, g_t: M\to\R$ be contact Hamiltonians. Let $\theta, \eta\in {\rm SH}_\ast(W)$. Then, 
    \[c(\theta\ast\eta, \mathbb V_{h\hat{\#}g})\leqslant c(\theta, \mathbb V_h) + c(\eta, \mathbb V_g) +3\cdot \max\left\{ \overline{\op{osc}}\: h, \overline{\op{osc}}\: g \right\}.\]
    Here, $\ast$ stands for the pair-of-pants product on ${\rm SH}_\ast(W).$
\end{theorem}

The proof of Theorem \ref{prop:triangle} will given in Section \ref{sec-proof-main} and it relies on the existence of the product map ${\rm HF}_\ast(h)\otimes {\rm HF}_\ast(g)\to {\rm HF}_\ast(h\hat{\#}g)$ from Section~\ref{sec-pp}. In particular, new analysis like the maximum principle on the pant-of-pants will be studied there.

\begin{cor} \label{cor-idemp} Let $(M, \xi)$ be a contact manifold with a Liouville filling $(W, \lambda)$ such that ${\rm SH}_*(W) \neq 0$. Then for any idempotent element $e \in {\rm SH}_*(W)$ (with respect to the pair-of-pants product $\ast$ on ${\rm SH}_*(W)$), we have $c(e, \mathbb V_0) \geq 0$. \end{cor}

\begin{proof} By Theorem \ref{prop:triangle}, take $h = g = 0$ and $\theta = \eta =e$, then 
\[ c(e\ast e, \mathbb V_{0\hat{\#}0})\leqslant c(e, \mathbb V_0) + c(e, \mathbb V_0) +3\cdot 0 = 2c(e, \mathbb V_0). \]
Meanwhile, $c(e\ast e, \mathbb V_{0\hat{\#}0}) = c(e, \mathbb V_0)$, therefore, we obtain the desired conclusion. \end{proof}

\begin{remark} Since the unit $\mathds{1} \in {\rm SH}_*(W)$ is an idempotent element, Corollary \ref{cor-idemp} refines the conclusion of Corollary \ref{cor-exist-tp} in the sense that it furthermore estimates the time shift of the existing translated point promised by Corollary \ref{cor-exist-tp} under the condition that ${\rm SH}_*(W) \neq 0$. \end{remark}

Fixing an idempotent element $e \in {\rm SH}_*(W)$, we can balance terms in the triangle inequality in Theorem \ref{prop:triangle} so that it becomes an inequality without extra term appearing. Define 
\begin{equation} \label{ref-dfn-c}
\tilde{c}(e, \mathbb V_h) : = c(e, \mathbb V_h) + 3\cdot \overline{\op{osc}}\: h.
\end{equation}
Then Theorem \ref{prop:triangle} yields that 
\begin{align*}
\tilde{c}(e, \mathbb V_{h \hat{\#} g}) & = c(e \ast e, \mathbb V_{h \hat{\#} g}) + 3 \cdot \overline{\op{osc}}\: (h \hat{\#} g)\\
& \leq c(e, \mathbb V_h) + c(e, \mathbb V_g) + 6\cdot \max\left\{ \overline{\op{osc}}\: h, \overline{\op{osc}}\: g \right\}\\
& \leq \left(c(e, \mathbb V_h) + 3 \cdot \overline{\op{osc}}\: h\right) + \left(c(e, \mathbb V_g) + 3 \cdot \overline{\op{osc}}\: g\right) = \tilde{c}(e, \mathbb V_{h}) + \tilde{c}(e, \mathbb V_{g}).
\end{align*}
In particular, for any fixed idempotent element $e \in {\rm SH}_*(W)$ and a contact Hamiltonian $h: [0,1] \times M \to \R$, the following limit exists,  
\begin{equation} \label{dfn-cont-qs}
\zeta_{\rm cont}(e, h) : = \lim_{k \to \infty} \frac{\tilde{c}\left(e, \mathbb V_{h\hat{\#} \cdots \hat{\#} h} \right)}{k}  = \lim_{k \to \infty} \frac{c\left(e, \mathbb V_{h\hat{\#} \cdots \hat{\#} h}\right)}{k}
\end{equation}
where observe that $\overline{\op{osc}}\: (h\hat{\#} \cdots \hat{\#} h) = \overline{\op{osc}}\: h$ by definition. This quantity $\zeta_{\rm cont}(e, -)$ will serve as a contact geometric analogue of the (partially) symplectic quasi-state defined by Entov-Polterovich in \cite{EP-rigidity}, where we call it a {\bf (partially) contact quasi-state} and will be explored further in \cite{DUZ-part2}, especially on its applications to study the rigidity of subsets in a (Liouville fillable) contact manifold (cf.~\cite{Borman-Zapolsky}). 

\medskip

The next result concerns about the well-definedness of the contact spectral invariant, when it descends to the universal cover of the contactomorphism group ${\rm Cont}(M, \xi)$. 

\begin{theorem}[Descent] \label{thm-descend} Let $(M, \xi)$ be a contact manifold with a Liouville filling $(W, \lambda)$ where $(\widehat{W}, \omega)$ denotes the completion of $(W, \lambda)$. Suppose either $\pi_2({\rm Cont}(M, \xi))$ or $\pi_1({\rm Ham}_c(\widehat{W}, \omega)$ is trivial, then the contact spectral invariant $c(a, -)$ is well-defined on $\widetilde{{\rm Cont}}_0(M, \xi)$ for any class $a \in {\rm SH}_*(W)$. \end{theorem}

Let us denote the resulting contact spectral invariant by $c(a, [\phi])$ for any class $[\phi] \in \widetilde{{\rm Cont}}_0(M, \xi)$. The triviality of $\pi_2({\rm Cont}(M, \xi))$ or $\pi_1({\rm Ham}_c(\widehat{W}, \omega)$ serves as two different situations where $c(a, -)$ descends. Therefore, the proof of Theorem \ref{thm-descend}, given in Section \ref{sec-more-wd}, will be divided into two parts. 

\begin{remark} It is worth mentioning that in the situation where $\pi_2({\rm Cont}(M, \xi))$ is trivial, the descended $c(a,-)$ is slightly better than the other one since the grading makes sense. On the other hand, these two situations are closely related to each other, since by the Biran-Giroux long exact sequence, we have a well-defined group homomorphism 
\[ \pi_2({\rm Cont}(M, \xi)) \to \pi_1({\rm Ham}_c(\widehat{W}, \omega)). \]
It will be interesting to explore when these two groups coincide. Another interesting point is that we are not aware of any example of a Liouville manifold $(\widehat{W}, \omega)$ where $\pi_1({\rm Ham}_c(\widehat{W}, \omega))$ is non-trivial. \end{remark}

Here is an immediate corollary of Theorem \ref{thm-descend}, related to the {\bf orderability} of a contact manifold, defined by \cite{EP00}. Recall that $(M, \xi)$ is non-orderable if there exists a contractible positive loop in the contactomorphism group ${\rm Cont}(M, \xi)$. In other words, there exists a positive loop, as a representative, of the unit class $\mathds{1} \in \widetilde{{\rm Cont}}_0(M, \xi)$. Here, ``positive'' means that the corresponding contact Hamiltonian is positive pointwisely.

\begin{cor} \label{cor-sh-order} Let $(M, \xi)$ be a contact manifold with a Liouville filling $(W, \lambda)$ where either $\pi_2({\rm Cont}(M, \xi))$ or $\pi_1({\rm Ham}_c(\widehat{W}, \omega))$ is trivial. Suppose ${\rm SH}(W) \neq 0$, then $(M, \xi)$ is orderable. \end{cor}

\begin{proof} Since ${\rm SH}(W) \neq 0$, the unit $1$ satisfies $c(1, \mathbb V_0)$ is finite number by either the proof of Corollary \ref{cor-exist-tp} or Corollary \ref{cor-idemp}. Assume $\phi = \{\phi_t\}_{t \in [0,1]}$ is a contractible positive loop in ${\rm Cont}(M, \xi)$. Then on the one hand, by Theorem \ref{thm-descend}, we have $c(1, \mathbb V_0) = c(1, [\mathds{1}]) = c(1, [\phi])$, where $\mathds{1}$ denotes the identity in ${\rm Cont}(M, \xi)$. On the other hand, since the corresponding contact Hamiltonian of $\phi$, denoted by $h: = h_t(\phi)$, satisfies $0< h_t(\phi)$ for any $t \in [0,1]$. By the compactness of $M$, there exists some $\delta>0$ such that $h \geq \delta$ pointwise. Then Theorem \ref{prop-shift} and Corollary \ref{cor-monotone} imply that 
\[ c(\mathds{1}, \mathbb V_0) < c(\mathds{1}, \mathbb V_0) + \delta = c(\mathds{1}, \mathbb V_\delta) \leq c(1, \mathbb V_h) = c(1, [\phi]).\]
This is a contradiction. \end{proof}

\begin{remark} A stronger version of Corollary \ref{cor-sh-order} was obtained in \cite[Theorem 1.18]{CCD-R}, where no topological hypothesis on the diffeomorphism is required and the conclusion there is about the strong orderability, a relative version of orderability based on Legendrian submanifolds. Also, \cite{AM} obtains a similar conclusion as in Corollary \ref{cor-sh-order} but via the non-vanishing of the Rabinowitz Floer homology, Our argument above  is closer to the one in \cite{AM}. \end{remark}

We will end this section with a few discussions and remarks. 

\medskip

(a) In this paper, we mainly work on the contact spectral invariant $c(a, \mathbb V_h)$ (Definition \ref{dfn-si-gap}), derived from homological visible information. From the gapped module $\mathbb V_h$, one can also read off homological invisible information. In term of the standard barcode language, this will be some data from finite length bars. This will be explored in the forthcoming work \cite{DUZ-part2}.

\medskip

(b) Despite the Definition \ref{dfn-si-gap}, which appears complicated (since we are using persistence module language to define spectral invariant), it is easily verified via \cite[Proposition 6.6]{UZ16} that $c(a, \mathbb V_h)$ can be equivalently defined or calculated by the following more classical way, 
\begin{equation}\label{dfn-si-classical}
c(a, \mathbb V_h) : = \inf\left\{\eta \in \R \,| \, a \in {\rm Im}(\iota_{\eta}: {\rm HF}_*(\eta \#h) \to {\rm SH}(W)) \right\}. 
\end{equation}
This also holds for the descended definition $c(a, \phi)$ for $\phi \in {\widetilde{{\rm Cont}}}_0(M, \xi)$.

\medskip 

(c) Recall that the symplectic spectral invariant, denoted by $c_{\rm symp}(a, \phi)$, where $\phi \in \widetilde{{\rm Ham}}(W)$, defined from the Hamiltonian Floer homology ${\rm HF}_*(H)$ with $\phi = \{\phi_H^t\}_{t \in [0,1]}$, satisfies a conjugate invariant property
\begin{equation} \label{symp-conj-invariant}
c_{\rm symp}(\psi^{-1}_*(a), \psi^{-1} \circ \phi \circ \psi) = c_{\rm symp}(a, \phi)
\end{equation}
for any symplectomorphism $\psi$. This is essentially due to the fact that the path $\{\psi^{-1} \circ \phi_H^t \circ \psi\}_{t \in [0,1]}$ is generated by the Hamiltonian $H \circ \psi$ and there is a natural filtration-preserving identification of groups ${\rm HF}_*(H)$ and ${\rm HF}_*(H \circ \psi)$. Different in the contact geometry set-up, if we have a contact isotopy $\phi = \{\phi_h^t\}_{t \in [0,1]}$, then it readily verified that $e^{\kappa(\psi)} \cdot (h \circ \psi)$ generates the path $\{\psi^{-1} \circ \phi_h^t \circ \psi\}_{t \in [0,1]}$. Then one verifies furthermore that $c(\psi^{-1}_*(a), \psi^{-1} \circ \phi \circ \psi) =  e^{\kappa(\psi)} c(a, \phi)$. This indicates that in general the conjugate invariant property of the contact does not hold in general. 

In fact, since our contact spectral invariant is $\R$-valued, one should {\it not} expect the validity of the conjugate invariant property; otherwise one would contradict a surprising result from Burago-Ivanov-Polterovich \cite{BIP} that no conjugate invariant norm that can be arbitrarily small exists on ${\rm Cont}(M, \xi)$. 

\medskip

(d) Inspired by Sandon's work \cite{San-non-squeezing} as well as related development in \cite{AM}, one can consider a $\Z$-valued function 
\[ \ceil{c(a, -)}: {\widetilde{{\rm Cont}}}_0(M, \xi) \to \Z.\]
In the forthcoming work \cite{DUZ-part2}, we will confirm that, under certain conditions on the Reeb flow, $\ceil{c(a, \phi)}$ is indeed conjugate invariant, which leads to the notation of a {\bf contact capacity} of a subset $U \subset M$. For applications, we aim to recover the contact non-squeezing phenomenon in \cite{U-selective} in a concise way. 

\section{A computational example of spectral invariant} \label{sec-ex} In this section, we compute the contact spectral invariant $c(a,\mathbb V_h)$ in a concrete case to confirm that this data is not always zero, {\bf even when $h =0$}.

\medskip

Consider $(S^*\mathbb S^2, \xi_{\rm can} = \ker \alpha|_{\rm can})$, the unit co-sphere bundle with respect to the canonical contact structure. It arise as the boundary of a Liouville filling - the unit co-disk bundle of 2-sphere $\mathbb S^2$ denoted by $D^*\mathbb S^2 = D^*_{g_{\rm std}} \mathbb S^2$, with respect to the standard round metric $g_{\rm std}$ on $\mathbb S^2$.  By \cite{Abbondandolo-Schwarz,Abbondandolo-Schwarz-2,Ziller}, the symplectic homology $D^\ast \mathbb{S}^2$ in $\mathbb{Z}_2$-coefficient is given by
\[{\rm SH}_k (D^\ast\mathbb{S}^2)\cong \left\{ \begin{matrix} \mathbb{Z}_2 &  k\in\{0,1\}\\ \mathbb{Z}_2 \bigoplus \Z_2 & k\geq 2\\ 0 & \text{otherwise.} \end{matrix}\right. \] 
We will compute the contact spectral invariant for $h =0$ (hence, ${\rm osc}(h) = 0$), as as explained in Example \ref{ex-lambda0} later we can view symplectic homology ${\rm SH}_\ast(D^\ast\mathbb{S}^2)$ as the following $\R$-parametrized persistence module valued at infinity, 
\[ \V_0 = \left\{{\rm HF}_\ast(S^\ast\mathbb{S}^2, \eta) = {\rm HF}_*(\eta)\right\}_{\eta \in \R}.\]
For simplicity, let us focus on $\eta \in \R_{\geq 0}$. It is well-known that closed Reeb orbits on $S^*\mathbb S^2$ precisely correspond to the closed geodesics on $(S^2, g_{\rm std})$, which implies that the Rabinowitz-Floer spectrum set ${\rm spec}(\mathcal A_h) = \{2\pi m \, |\, m \in \N\}$.

By \cite[Proposition 4.4]{U-Viterbo}, we have
\[{\rm HF}_k(2\pi m+\varepsilon)\cong\left\{ \begin{matrix}  \mathbb{Z}_2 & k\in\{0,1,2m+1, 2m+2\} \\ \mathbb{Z}_2 \oplus \Z_2 & 2\leq k \leq 2m\\ 0 & \text{otherwise}\end{matrix}\right.\]
for $m\in\mathbb{N}$ and any $\varepsilon\in(0,2\pi)$. Moreover, the canonical map (that comes from the composition of Floer continuation maps for consecutive $m$), 
\begin{equation} \label{can-map}
{\rm HF}_k(2\pi m+\varepsilon)\to {\rm SH}_k(D^\ast \mathbb{S}^2)
\end{equation}
is an isomorphism if $2m> k-1$. When $k =0$, the canonical map (\ref{can-map}) is an isomorphism starting from $m =0$. When $k = 1$, the canonical map (\ref{can-map}) is an isomorphism starting from $m =1$. Therefore, denote by $a_0$ and $a_1$ the generators of ${\rm SH}_0(D^\ast \mathbb{S}^2)$ and ${\rm SH}_1(D^\ast \mathbb{S}^2)$, respectively. We have the following easy computation of the barcode. 
\begin{itemize}
\item[(0)] The degree $0$ barcode $\V_0$ is $[\ell, +\infty)$ where $\ell$ is some non-positive constant. This implies that $c(a_0, 0) \leq 0$. 
\item[(1)] The degree $1$ barcode $\V_0$ is $[2 \pi, +\infty)$. This implies that $c(a_1, 0) = 2 \pi$. 
\end{itemize}
Note that for both (0) and (1), the perturbation constant $\ep>0$ is taken to be arbitrarily small. 
 
Now, observe that for persistence module $\V_0$ for degree $k \geq 3$, there is a generator born starting from $m = \left\lfloor\frac{k-1}{2} \right\rfloor$, while another generator born starting from $m = \left\lfloor\frac{k-1}{2} \right\rfloor + 1$. Both of them are homologically essential since by (\ref{can-map}) is an isomorphism when $m \geq \left\lfloor\frac{k-1}{2} \right\rfloor + 1$. Therefore, the barcode of $\V_0$ at degree $k \geq 3$ is the union of the following two intervals, 
\begin{equation} \label{bar-k}
\left[2 \pi \cdot \left\lfloor\frac{k-1}{2} \right\rfloor, +\infty\right) \,\,\,\,\mbox{and}\,\,\,\, \left[2 \pi \cdot \left(\left\lfloor\frac{k-1}{2} \right\rfloor+1\right), +\infty\right).
\end{equation}
Denote by $a_k$ and $b_k$ the generators corresponding to the first interval and the second interval in (\ref{bar-k}), respectively. Then for $k \geq 3$, 
\begin{equation} \label{cis-k}
c(a_k, 0)= 2\pi\cdot\left\lfloor \frac{k-1}{2} \right\rfloor \,\,\,\,\mbox{and}\,\,\,\, c(b_k,0)=  2\pi\cdot\left\lfloor \frac{k-1}{2} \right\rfloor+ 2\pi.
\end{equation}
The only case that needs some special care is when $k =2$. In this case, the generator $a_2 \in {\rm HF}_2(\ep)$ where $m=0$, which only implies that $c(a_2, 0) \leq 0$ and it is easy to verify that $c(b_k,0)= 2\pi$. 


\section{Gapped modules and its restrictions}\label{sec-gap-mod}

Recall that a persistence $\k$-module, in a general sense, is an $\mathcal I$-parametrized family of $\k$-modules denoted by $\mathbb V = \{V_i\}_{i \in \mathcal I}$, where the index set $\mathcal I$ is a subset of $\R$, together with a certain family of $\k$-linear maps. Explicitly, by the partial order $\leq$ of $\R$, for each $i \leq j$ in $\mathcal I$, there is a $\k$-linear map $\iota_{i,j}: V_i \to V_j$ satisfying $\iota_{j,k} \circ \iota_{i,j} = \iota_{i,k}$ (and $\iota_{i,i} = \mathds{1}_{V_i}$). Suppose $\mathcal I$ is totally ordered, then by \cite{C-B-decomposition}, the persistence $\k$-module $\mathbb V$ admits a decomposition as follows, 
\begin{equation} \label{normal form}
\mathbb V = \bigoplus_{I \subset B(\mathbb V)} \k_I 
\end{equation}
where $B(\mathbb V)$, called the barcode of $\mathbb V$, is a collection of some intervals $I$ in $\mathcal I$. Here, an interval $I$ of $\mathcal I$ means that for any $i, k \in I$ and $j \in I$ with $i \leq j \leq k$, then $j \in \mathcal I$. Moreover, $\k_I$ denotes the trivial rank-1 bundle over the interval $I$. For more details on persistence module theory, see \cite{CZ05,CC-SGGO-prox-09,CdeSGO16,PRSZ20}.

\subsection{A general theory on gapped module} Fix a scalar $\lambda \geq 0$ and subset $\mathcal I \subset \R$. For $s, t \in \mathcal I$, denote by $s \leq_{\lambda} t$ the relation either $s =t$ or $s \leq t -\lambda$. This relation defines a partial order on $\mathcal I$, which may not be a total order (since there is no relation between $s$ and $s + \ep$ if $s, s+\ep \in \mathcal I$ with $\ep < \lambda$, for instance). Then we have the following key definition.
 
\begin{dfn} \label{dfn-gap-mod} Fix a field $\k$ and a scalar $\lambda \geq 0$. A {\rm $\lambda$-gapped $\k$-module} $\mathbb V$ consists of the following data 
\[ \mathbb V= (\{V_t\}_{t \in \mathcal I}, \{\iota_{s,t}: V_s \to V_t\}_{s \leq_{\lambda} t}\}) \]
where $V_t$ is a finite dimensional $\k$-module, $\iota_{t,t} = \mathds{1}$ and $\iota_{s,t} \circ \iota_{r,s} = \iota_{r,t}$ for $r \leq_{\lambda} s \leq_{\lambda} t$. \end{dfn}

As mentioned in the introduction (see the end of Section \ref{sec-construction}), a persistence $\k$-module can be defined over a partially order set, say $(\mathcal I, \leq_{\lambda})$ as above, but the decomposition (into interval-type module) result as in (\ref{normal form}) only works for a totally ordered parametrization set. To this end, to extract information from $(\mathcal I, \leq_{\lambda})$ in terms of the barcode, we will restrict our parametrization set from $\mathcal I$ to certain discrete sequences $\mathfrak a = \{\cdots, \eta_i, \eta_{i+1}, \cdots\}$.

\begin{dfn} \label{dfn-restriction} Fix a field $\k$, a scalar $\lambda > 0$, and a $\lambda$-gapped $\k$-module $\mathbb V$. A {\rm $\lambda$-gapped restriction} $\mathbb V(\mathfrak a)$ is a persistence $\k$-module index by a discrete subset 
\[ \mathfrak a  = \{\cdots, \eta_i, \eta_{i+1}, \cdots\} \subset \mathcal I\]
such that $\eta_{i+1} - \eta_i \geq \lambda$ for each $i \in \Z$. Moreover, we have the following refinements. 
\begin{itemize}
\item[(1)] The index sequence $\mathfrak a$ is called {\rm almost optimal} if $\eta_{i+1} - \eta_i = \lambda + {\o}^{\mathfrak a}_i(1)$ where 
\[ 0 \leq {\o}^{\mathfrak a}_i(1) \leq \frac{1}{C_{\mathfrak a}\cdot 2^{|i|+1}} \,\,\,\,\mbox{for some constant $C_{\mathfrak a} >> \frac{1}{\lambda}$}\]
 for each $i \in \Z$. Then $\mathbb V(\mathfrak a)$ is called almost optimal. 
 \item[(2)] The index sequence $\mathfrak a$ is called {\rm normalized} if it is almost optimal and $\eta_0 \in [0, \lambda]$. Then $\mathbb V(\mathfrak a)$ is called normalized.  
 \end{itemize}
 \end{dfn}
 
 Here are some examples of gapped $\k$-modules and their gapped restrictions. 
 
 \begin{ex} Any $1$-gapped $\k$-module $\mathbb V$ over $\mathcal I = \R$ admits an almost optimal $1$-gapped restrictions with all perturbation constants ${\o}_i(1) =0$. Indeed, one simply take 
\[ \mathfrak a = \{\cdots 0, 1, 2, \cdots\}. \]
If moreover, one assigns the $i$-th element in $\mathfrak a$ by $\eta_i = i$, then $\mathbb V(\mathfrak a)$ is a normalized restriction, and it is a standard $\Z$-indexed persistence $\k$-module. Note that not every $\lambda$-gapped $\k$-module admits a normalized restriction, not even admitting an almost optimal restriction. For instance, if $\mathcal I = \{0, 2, 2^2, \cdots, 2^n, \cdots\}$, then a $1$-gapped $\k$-module $\mathbb V$ over $\mathcal I$ does not admit any $1$-gapped almost optimal restriction, essentially because $\mathcal I$ is too sparse. 
\end{ex}

\begin{ex} \label{ex-chfh-gap} For any Liouville fillable contact manifold $(M, \xi = \ker \alpha)$ and a contact Hamiltonian $h: [0,1] \times M \to \R$, via the contact Hamiltonian Floer homology, one can construct an ${\rm osc}(h)$-gapped $\k$-module over $\R \backslash \mathcal S_h$ as in Section \ref{sec-construction}, denoted by $\mathbb V_h$, where $\mathcal S_h$ is a discrete subset of $\R$ defined by (\ref{dfn-RFspec}). Moreover, it admits a normalized $\lambda$-gapped restriction. Indeed, since $\eta$ can be chosen away from a , one can pick an arbitrary 
\[ \eta_0 \in [0, {\rm osc}(h)) \cap (\R \backslash \mathcal S_h) \]
and set $\eta_i = \eta_0 + i \cdot {\rm osc}(h)$ with a small perturbation ${\o}_i(1)$ to avoid $\mathcal S_h$. Then this $\mathbb V_h(\mathfrak a)$ with $\mathfrak a = \{\cdots \eta_0, \eta_1, \cdots\}$ is the desired restriction. \end{ex}

\begin{ex} \label{ex-lambda0} When $\lambda =0$, the relation $s \leq_{\lambda = 0} t$ reduces to the standard relation $s \leq t$. Then any standard persistence $\k$-module over $\mathcal I$ is a $0$-gapped $\k$-module. For any index sequence $\mathfrak a \subset \mathcal I$, the corresponding $\mathbb V(\mathfrak a)$ is a $0$-gapped restriction. Moreover, suppose $\mathcal I = \R$, then as long as elements $\eta_i$ in $\mathfrak a$ are different but sufficiently close consecutively, $\mathbb V(\mathfrak a)$ is almost optimal with ${\o}_i(1) =0$ for all $i \in \Z$. For instance, one can take $\mathfrak a = \{\cdots, - \ep, 0, \ep, \cdots\}$ for small $\ep>0$. When $\ep \to 0$, the restriction $\mathbb V(\mathfrak a)$ approximates to the original persistence $\k$-module $\mathbb V$ over $\R$. \end{ex}

\noindent Here are some basic observations on gapped $\k$-modules. 

\medskip

\noindent (i) By the general theory on persistence modules as explained at the beginning of this section, for every $\lambda$-gapped $\k$-module, its restriction $\mathbb V(\mathfrak a)$ admits a barcode $B(\mathbb V(\mathfrak a))$, where the endpoints of bars come from $\eta_i$'s in $\mathfrak a$. 

\medskip

\noindent (ii) Any $\lambda$-gapped $\k$-module $\mathbb V$ is automatically a $\lambda'$-gapped $\k$-module if $\lambda' \leq \lambda$. However, for a fixed $\lambda$-gapped $\k$-module $\mathbb V$, it is usually difficult to compare two $\lambda$-gapped restrictions $\mathbb V(\mathfrak a)$ and $\mathbb V(\mathfrak b)$, especially when these index sequences $\mathfrak a$ and $\mathfrak b$ are very different to each other.

\medskip

\noindent (iii) For two almost optimal $\lambda$-gapped restrictions $\mathbb V(\mathfrak a)$ and $\mathbb V(\mathfrak b)$. They are comparable in the sense that by a shift at most up to $\lambda$, the index sequence $\mathfrak a$ and $\mathfrak b$ coincide (modulo the perturbation constants from ${\o}_i(1)$). However, note that the index-shifted sequence $\mathfrak a[n]$, given by defining the $i$-th element as the $\eta_{i+n} \in \mathfrak a$, obviously defines a $\lambda$-gapped $\k$-module 
\begin{equation} \label{shift}
\mathbb V(\mathfrak a[n]) \,\,\,\,\mbox{which can also be denoted by $\mathbb V(\mathfrak a)[n\lambda]$}, 
\end{equation}
since $B(\mathbb V(\mathfrak a[n]))$ is a shift of $B(\mathbb V(\mathfrak a))$ by $n\lambda$. This means that even though $\mathfrak a$ and $\mathfrak b$ are comparable index sequences, $\mathbb V(\mathfrak a)$ and $\mathbb V(\mathfrak b)$ may differ a lot in terms of the barcodes $B(\mathbb V(\mathfrak a))$ and $B(\mathbb V(\mathfrak b))$. 

\medskip

For normalized $\lambda$-gapped restrictions of a fixed $\lambda$-gapped $\k$-module, we have the following stability result.  
Recall that persistence $\k$-modules are quantitatively comparable via the interleaving distance $d_{\rm inter}$, defined as a certain shifted persistence isomorphism. It essentially takes advantage that the parameterization set $\mathcal I \subset \R$ of a persistence module can be shifted.

\begin{prop} \label{prop-change-function}
Fix a scalar $\lambda > 0$ and a $\lambda$-gapped $\k$-module. Then for any normalized $\lambda$-gapped restrictions $\mathbb V(\mathfrak a)$ and $\mathbb V(\mathfrak b)$, we have $d_{\rm inter}(\mathbb V(\mathfrak a), \mathbb V(\mathfrak b)) \leq 3\lambda$. \end{prop}

\begin{proof} Set $\mathfrak a = \{\cdots, \eta_0, \eta_1, \cdots\}$ and $\mathfrak b = \{\cdots, \tau_0, \tau_1, \cdots\}$. By definition, without loss of generality, assume that $0 \leq \eta_0 \leq \tau_0 < \lambda$. Without of loss generality, let us consider $\ell \in \Z_{\geq 0}$. According to Definition \ref{dfn-restriction}, we have
\begin{align*}
\eta_{\ell} & = \eta_0 + \ell \lambda + \sum_{0 \leq i \leq \ell} {\o}^{\mathfrak a}_i(1) \\
& \leq \tau_0 + (\ell+1) \lambda \\
& \leq \left(\tau_0 + (\ell+2) \lambda +  \sum_{0 \leq i \leq \ell+2} {\o}^{\mathfrak b}_i(1)\right) - \lambda  = \tau_{\ell+2} - \lambda.
\end{align*}
Therefore, by Definition \ref{dfn-gap-mod}, there exists a morphism $\iota_{\eta_\ell \,\tau_{\ell+2}}: V_{\eta_\ell} \to V_{\tau_{\ell+2}}$. Similarly, we have 
\begin{align*}
\tau_{\ell} &= \tau_0 + \ell \lambda +  \sum_{0 \leq i \leq \ell} {\o}^{\mathfrak b}_i(1)\\
& \leq (\eta_0 + \lambda) + (\ell+1) \lambda \\
& \leq \left(\eta_0 + (\ell+3)\lambda + \sum_{0 \leq i \leq \ell+3} {\o}^{\mathfrak a}_i(1)\right) - \lambda =  \eta_{\ell + 3} - \lambda. 
\end{align*}
Therefore, by Definition \ref{dfn-gap-mod}, there exists a morphism $\iota_{\tau_{\ell} \,\eta_{\ell+3}}: V_{\tau_\ell} \to V_{\eta_{\ell+3}}$. In terms of the notations, $\mathbb V(\mathfrak a)_{\eta_\ell} = V_{\eta_\ell}$ and $(\mathbb V(\mathfrak a)[n\lambda])_{\eta_\ell} = V_{\eta_{\ell+n}}$, similarly to $\mathbb V(\mathfrak b)$. Then we have the following commutative diagram,
\[\xymatrixcolsep{5pc}\ \xymatrix{ \mathbb V(\mathfrak a)_{\eta_{\ell}} \ar[r]^-{\iota_{\eta_{\ell} \, \tau_{\ell+2}}} \ar@/_2.5pc/[rr]^-{\iota_{\eta_{\ell} \,\eta_{\ell+5}}} & \left(\mathbb V(\mathfrak b)[2\lambda]\right)_{\tau_\ell} \ar[r]^-{\iota_{\tau_{\ell} \, \eta_{\ell+3}}[2\lambda]} & \left(\mathbb V(\mathfrak a)[5\lambda]\right)_{\eta_{\ell}}}\]
and symmetric diagram. Then by definition we have $\mathbb V(\mathfrak a)$ and $\mathbb V(\mathfrak b)$ are $3\lambda$-interleaved, which implies the desired conclusion. 
\end{proof}

\begin{remark} \label{rmk-eta-valued-inter} (1) We need the $3\lambda$-interleaving relation (instead of a $2\lambda$-interleaving relation) precisely to taking care of the adjustment constant ${\o}(1)$. (2) When $\lambda =0$, up to the limit process as in Example \ref{ex-lambda0} if possible (for instance, the parametrization set $\mathcal I = \R$), there will be only one normalized sequence in the limit which implies that $d_{\rm inter}(\mathbb V(\mathfrak a), \mathbb V(\mathfrak b)) = 2 \lambda$.  \end{remark}


Here is a direct corollary of Proposition \ref{prop-change-function} and (2) of Remark \ref{rmk-eta-valued-inter}. Recall that a quantitative comparison between barcodes is called the bottleneck distance denoted by $d_{\rm bottle}$, which transfers the theoretically incomputable distance $d_{\rm inter}$ into a combinatorial type distance, in particular, easier to compute. 

\begin{cor} \label{cor-long-bar} Fix a scalar $\lambda \geq 0$, a $\lambda$-gapped $\k$-module, and a $\lambda$-gapped module $\mathbb V$. Then any two $\lambda$-gapped restrictions $\mathbb V(\mathfrak a)$ and $\mathbb V(\mathfrak b)$ satisfy 
\[ d_{\rm bottle}(B(\mathbb V(\mathfrak a)), B(\mathbb V(\mathfrak b))) \leq 3\lambda. \]
In particular, the cardinalities of infinite length bars in the barcodes $B(\mathbb V(\mathfrak a))$ and $B(\mathbb V(\mathfrak b))$ are the same, and there is a one-to-one correspondence between them with left endpoints shifted up to $3\lambda$.  \end{cor}

\begin{proof} This directly comes from the standard isometry theorem in persistence module theory, say, the main result of \cite{BL15}, and the definition of $d_{\rm bottle}$. \end{proof}

Next, we consider two gapped modules with possibly two different gaps. The potential problem is that two gapped modules may be parametrized by different sets, so it is not obvious how to compare them via their gapped restrictions. The following definition provides a possible approach. 

\begin{dfn} \label{dfn-gap-inter} Fix scalars $\lambda, \lambda' \in \R_{> 0}$. Suppose $\mathbb V$ is a $\lambda$-gapped $\k$-module indexed by $\mathcal I$ and $\mathbb W$ is a $\lambda'$-gapped $\k$-module indexed by $\mathcal J$, where $\mathcal I, \mathcal J$ are subsets of $\R$. These gapped $\k$-modules $\mathbb V$ and $\mathbb W$ are called {\rm $\delta$-interleaved} for a scalar $\delta \in \R_{\geq 0}$ either $\delta=0$ (only for $\mathbb V = \mathbb W$) or $\delta \geq \max\{\lambda, \lambda'\}$ if for any almost optimal $\delta$-gapped restrictions $\mathbb V(\mathfrak a)$ and $\mathbb W(\mathfrak a)$ where $\mathfrak a = \{\cdots, \eta_0, \eta_1, \cdots\} \subset \mathcal I \cap \mathcal J$, we have morphisms 
\[ \phi = \left\{\phi_i:  \mathbb V(\mathfrak a)_{\eta_{i}} \to \left(\mathbb W(\mathfrak a)[\delta]\right)_{\eta_i} =\mathbb W(\mathfrak a)_{\eta_{i+1}}\right\}_{i \in \Z} \]
and 
\[ \psi = \left\{\psi_i:  \mathbb W(\mathfrak a)_{\eta_{i}} \to \left(\mathbb V(\mathfrak a)[\delta]\right)_{\eta_i} =\mathbb V(\mathfrak a)_{\eta_{i+1}}\right\}_{i \in \Z} \]
such that the following diagrams commute:
\[\xymatrixcolsep{5pc}\ \xymatrix{ \mathbb V(\mathfrak a)_{\eta_{i}} \ar[r]^-{\phi_i} \ar@/_2.5pc/[rr]^-{\iota^{\mathbb V}_{\eta_{i} \,\eta_{i+2}}} & \left(\mathbb W(\mathfrak a)[\delta]\right)_{\eta_i} \ar[r]^-{\psi_i[\delta]} & \left(\mathbb V(\mathfrak a)[2\delta]\right)_{\eta_{i}}}\]
and 
\[\xymatrixcolsep{5pc}\ \xymatrix{ \mathbb W(\mathfrak a)_{\eta_{i}} \ar[r]^-{\psi_i} \ar@/_2.5pc/[rr]^-{\iota^{\mathbb W}_{\eta_{i} \,\eta_{i+2}}} & \left(\mathbb V(\mathfrak a)[\delta]\right)_{\eta_i} \ar[r]^-{\phi_i[\delta]} & \left(\mathbb W(\mathfrak a)[2\delta]\right)_{\eta_{i}}}\]
for any $i \in \Z$. 
\end{dfn}

\begin{remark} The existence of a $\delta$-interleaving relation for some finite $\delta \geq 0$ between $\mathbb V$ and $\mathbb W$ automatically implies there exists an almost optimal sequence $\mathfrak a \subset \mathcal I \cap \mathcal J$. In general it is possible that $\mathcal I \cap \mathcal J = \emptyset$, but $\mathbb V$ and $\mathbb W$ are close to each other in terms of the interleaving relation. However, we will not encounter this situation in this paper. \end{remark} 

Note that including $\delta=0$ is to make sure that any $\lambda$-gapped $\k$-module $\mathbb V$ and itself are $0$-interleaved, where $\phi = \psi = \mathds{1}$ (instead of being $\lambda$-interleaved which is not a sharp estimation). Of course one can assume $\mathfrak a$ is normalized (so are $\mathbb V(\mathfrak a)$ and $\mathbb W(\mathfrak b)$) since a shift of $\mathfrak a$ does not affect the diagrams in Definition \ref{dfn-gap-inter}.

\subsection{Contact spectral invariant} Given a persistence $\k$-module $\mathbb V$ (parametrized by $\R$ or a discrete subset $\mathcal I \subset \R$) and an element $a \in (\mathbb V)_{\infty}$, in the modern language of barcode, the spectral
invariant of $a$ can be read off directly from the information of certain bars. Before giving the explicit procedure, recall that infinite length bars in $B(\mathbb V)$ correspond to a (not necessarily unique) basis of the $\k$-module $(\mathbb V)_{\infty}$, denoted by $e = \{e_1, ... , e_n\}$ where $n = \dim_{\k} (\mathbb V)_{\infty}$. Then under this basis $e$, the element $a$ admits the following linear combination 
\[ a = x_1 e_1 + \cdots x_n e_n, \,\,\,\,\mbox{where $x_i \in \k$}. \]
Introduce the notation $\mathfrak s(e_i)$ as the left endpoint of the infinite length bar corresponding to the basis element $e_i$ (called a filtration spectrum as in Definition 5.2 in \cite{UZ16}). Then the spectral invariant of $a$, denoted by $c(a)$, can be defined or computed (see Proposition 6.6 in \cite{UZ16}) by 
\begin{equation} \label{def-si}
c(a, \mathbb V) = \max\{\mathfrak s(e_i) \, | \, x_i \neq 0\}. 
\end{equation}
A general result, say Theorem 7.1 in \cite{UZ16}, proves that the multi-set $\{\mathfrak s(e_i)\}_{i}$ is unique for a fixed $\mathbb V$ even though basis $e$ is not uniquely determined in general.

\medskip

Now, in the case of a $\lambda$-gapped $\k$-module $\mathbb V$, we will define the spectral invariant of $a$ with the help of $\lambda$-gapped restrictions. 

\begin{dfn} \label{dfn-si-gap} Fix a scalar $\lambda \in \R_{>0}$. Suppose $\mathbb V$ is a $\lambda$-gapped $\k$-module indexed by $\mathcal I$ where $+\infty$ is an accumulated point of $\mathcal I$. For any $a \in \varinjlim_{i \in I} (\mathbb V)_{i}$, define its spectral invariant by 
\[ c(a, \mathbb V): = \inf \left\{c(a, \mathbb V(\mathfrak a)) \,\mbox{as defined in (\ref{def-si})} \,\bigg| \, \begin{array}{l} \mbox{$\mathfrak a$ is a $\lambda'$-normalized sequence as in} \\ \mbox{(2) of Definition \ref{dfn-restriction} for any $\lambda' \geq \lambda$} \end{array}\right\}. \]
As a convention, define $c(0, \mathbb V) = +\infty$. 
\end{dfn}

Note that Corollary \ref{cor-long-bar} implies that when $a \neq 0$, the spectral invariant $c(a, \mathbb V)$ defined in Definition \ref{dfn-si-gap} is a finite number. Also, observe that when $\lambda = 0$ and $\mathcal I = \R$, the definition $c(a, \mathbb V)$ in Definition \ref{dfn-si-gap} coincides with the computational definition in (\ref{def-si}). Finally, it is readily verified that when $\mathcal I$ is a dense subset of $\R$, $c(a, \mathbb V)$ can be equivalently defined by using $\lambda$-normalized restrictions $\mathbb V(\mathfrak a)$. 

\begin{remark} \label{rmk-generalized-csi} For a $\lambda$-gapped $\k$-module $\mathbb V$, here is an alternative way to compute or define $c(a, \mathbb V)$ via almost optimal sequences. For any $\lambda$-almost optimal restriction $\mathbb V(\mathfrak a)$ for an index sequence $\mathfrak a: \Z \to \mathcal I$ (which is not necessarily $\lambda$-normalized), there exists a unique integer $m^{\mathfrak a} \in \Z$ such that $\mathfrak a - m^{\mathfrak a} \lambda: \Z \to \mathcal I$ defined by $(\mathfrak a - m^{\mathfrak a} \lambda)(i) = \mathfrak a(i) - m^{\mathfrak a} \lambda$ is $\lambda$-normalized (up to the adjustment of the constants ${\o}^{\mathfrak a}_i(1)$ in (1) of Definition \ref{dfn-restriction}). For instance, one can take $m^{\mathfrak a} = \floor{\frac{\mathfrak a(0)}{\lambda}}$. Then define the following quantity, 
\[ \overline{c}(a, \mathbb V): = \inf \left\{c(a, \mathbb V(\mathfrak a)) - m^{\mathfrak a}\lambda'\,\bigg| \, \begin{array}{l} \mbox{$\mathfrak a$ is a $\lambda'$-almost optimal sequence as in} \\ \mbox{(1) of Definition \ref{dfn-restriction} for any $\lambda' \geq \lambda$} \end{array}\right\}. \]
Clearly, if $\mathfrak a$ is already $\lambda$-normalized, then $m^{\mathfrak a} = 0$, and thus $\overline{c}(a, \mathbb V) \leq c(a, \mathbb V)$. Moreover, up to $\ep \geq 0$, suppose $\overline{c}(a, \mathbb V)$ is obtained via a $\lambda'$-almost optimal sequence $\mathfrak a: \Z \to \mathcal I$, then $\mathfrak a - m^{\mathfrak a} \lambda'$ by construction is $\lambda'$-normalized, where $\mathbb V(\mathfrak a - m^{\mathfrak a}\lambda')$ serves as a candidate in the definition of $c(a, \mathbb V)$ in Definition \ref{dfn-si-gap}. Therefore, we have 
\[ c(a, \mathbb V) \leq c(a, \mathbb V(\mathfrak a - m^{\mathfrak a}\lambda')) = c(a, \mathbb V(\mathfrak a)) - m^{\mathfrak a} \lambda' = \overline{c}(a, \mathbb V), \]
where the first equality comes from the standard translation property of the spectral invariant of class $a$ from persistence module $\mathbb V(\mathfrak a)$. In this way, we confirm that $c(a, \mathbb V) = \overline{c}(a, \mathbb V)$.\end{remark}

The most important property of $c(a,\mathbb V)$ in this paper is the following stability result. 

\begin{prop} \label{prop-alg-stability} Fix scalars $\lambda, \lambda' \in \R_{\geq 0}$ and $\delta_0 \geq \max\{\lambda, \lambda'\}$. Suppose $\mathbb V$ is a $\lambda$-gapped $\k$-module indexed by $\mathcal I$ and $\mathbb W$ is a $\lambda'$-gapped $\k$-module indexed by $\mathcal J$, where $\mathcal I, \mathcal J$ are dense subsets of $\R$. If $\mathbb V$ and $\mathbb W$ are $\delta$-interleaved as in Definition \ref{dfn-gap-inter} for any $\delta \geq \delta_0$, then for any $a \in \varinjlim_{i \in I} (\mathbb V)_{i} \cap \varinjlim_{j \in J} (\mathbb W)_{j}$, the contact spectral invariants of $a$ as in Definition \ref{dfn-si-gap} satisfy 
\[ |c(a, \mathbb V) - c(a, \mathbb W)| \leq C \delta_0 \]
for some constant $C \geq 0$. More explicitly, when $\mathbb V \neq \mathbb W$, we have $C = 2$; when $\mathbb V = \mathbb W$, we have the obvious $C = 0$. 
\end{prop}

\begin{proof} According to the definition of $c(a, \mathbb V)$ in Definition \ref{dfn-si-gap}, for any $\ep>0$, there exists a $\lambda$-normalized sequence $\mathfrak a \subset \mathcal I$ and the associated $\lambda$-gapped restriction $\mathbb V(\mathfrak a)$ such that $c(a, \mathbb V(\mathfrak a))  \leq c(a, \mathbb V) + \ep$. By the definition in (\ref{def-si}), suppose $c(a, \mathbb V(\mathfrak a)) = \mathfrak s(e_{i_*})$, as the left endpoint of an interval $I \in B(\mathbb V(\mathfrak a))$ corresponding to the basis element $e_{i_*}$. In particular, $\mathfrak s(e_{i_*}) \in \mathfrak a$. Now, remove certain terms in $\mathfrak a$ but keep $\mathfrak s(e_{i_*})$ so that the resulting new sequence denoted by $\mathfrak b$ is $\delta$-gapped for some $\delta \in [\delta_0, 2 \delta_0]$. Note that this can obtained since $\lambda \leq \delta_0$ and the requested $\delta$ can be chosen, for instance, as
\[ \delta = k \lambda \,\,\,\,\mbox{where $k = \lfloor\frac{\delta_0}{\lambda} \rfloor +1$}. \]
Moreover, since both $\mathcal I$ and $\mathcal J$ are dense in $\R$, up to small perturbations of the items in $\mathfrak b$, we can assume that $\mathfrak b \subset \mathcal J$ as well. In this way, we obtain almost optimal $\delta$-gapped restrictions $\mathbb V(\mathfrak b)$ and $\mathbb W(\mathfrak b)$. By our assumption, persistence modules $\mathbb V(\mathfrak b)$ and $\mathbb W(\mathfrak b)$ are $\delta$-interleaved. 

Next, we claim that $c(a, \mathbb V(\mathfrak b)) = c(a, \mathbb V(\mathfrak a))$ for the $\delta$-gapped sequence $\mathfrak b$. Indeed, removing terms from $\mathfrak a$ as above can only result in an increasing of the spectral invariant of $a$. Suppose $c(a, \mathbb V(\mathfrak b)) > c(a, \mathbb V(\mathfrak a))$, then $c(a, \mathbb V(\mathfrak b)) = \mathfrak s(e_{j_*})$ for some $j_* \neq i_*$, since $\mathfrak s(e_{i_*})$ is not removed from $\mathfrak a$ by our construction and the infinite length bars in $B(\mathbb V(\mathfrak a))$ and those in $B(\mathbb V(\mathfrak b))$ are one-to-one corresponded. In particular, the bar $I = [\mathfrak s(e_{i_*}), \infty) \in B(\mathbb V(\mathfrak a))$ corresponds to $I' = [\mathfrak s(e_{i_*}), \infty) \in B(\mathbb V(\mathfrak a))$. By definition in (\ref{def-si}), the bar $J' = [\mathfrak s(e_{j_*}), \infty) \in B(\mathbb V(\mathfrak a))$ corresponds to some $J \in B(\mathbb V(\mathfrak a))$ where its left endpoint is no greater than $\mathfrak s(e_{i_*})$. In particular, the value $\mathfrak s(e_{i_*})$ lies in the interval $J$. This provides a contradiction since again $\mathfrak s(e_{i_*}) \in \mathfrak a$ remains and it is impossible to obtain $J'$ from $J$ by removing terms from $\mathfrak a$. 

Finally, since $\mathbb V(\mathfrak b)$ and $\mathbb W(\mathfrak b)$ are $\delta$-interleaved, we have $|c(a, \mathbb V(\mathfrak b)) - c(a, \mathbb W(\mathfrak b))| \leq \delta$. 
\begin{align*}
c(a, \mathbb W(\mathfrak b)) & \leq c(a, \mathbb V(\mathfrak b)) + \delta \\
& = c(a, \mathbb V(\mathfrak a)) +\delta \leq c(a, \mathbb V) + \ep + \delta.
\end{align*}
This implies that $c(a, \mathbb W) \leq  c(a, \mathbb V) + \ep + 2 \delta_0$. Thus we obtain the desired conclusion by switching $\mathbb V$ and $\mathbb W$ and let $\ep \to 0$. 
\end{proof}

\section{Pair-of-pants product} \label{sec-pp}

In this section, we prove a maximum principle that enables us to define the pair-of-pants product ${\rm HF}_\ast(h^{a_1})\otimes {\rm HF}_\ast(h^{a_2})\to {\rm HF}_\ast(h^b)$ for contact Hamiltonian Floer homology. This product is associated with a triple of admissible contact Hamiltonians $h^{a_1}, h^{a_2},$ and $h^b$ such that the following holds:
\begin{enumerate}
    \item there exists $\delta>0$ such that $h^{a_j}_t=0$ for $t\in  (-\delta, \delta)$ and $j\in\{1,2\}.$
    \item there exists $\delta>0$ such that $h^{b}_t=0$ for $t\in  (-\delta, \delta) \cup (\frac{1}{2}-\delta, \frac{1}{2}+\delta)$.
    \item $h^{a_1} \hat{\#} h^{a_2} \leqslant h^b$ where $h^{a_1} \hat{\#} h^{a_2} := \left\{  \begin{matrix} h^{a_1}_{2t}, & t \in \left[0, \frac{1}{2}\right] \\  h^{a_2}_{2t-1}, & t \in \left[\frac{1}{2}, 1 \right]       \end{matrix}   \right..$
\end{enumerate}
The pair-of-pants product is defined by counting solutions of the Floer equation parametrized by a sphere with three points removed (i.e., by a pair-of-pants).  The precise definition of the pair-of-pants product is given in Section~\ref{sec:defPPP}. In Sections \ref{sec:pants} and \ref{sec:productdata}, we introduce the necessary objects for the definition of the pair-of-pants product, namely the pair-of-pants with a slit and product data. The Section~\ref{sec:compactnessPPP} proves the maximum principle (as a Theorem by its own) that implies the compactness of the moduli spaces and thus justifies the definition.  

\subsection{Pair-of-pants with a slit}\label{sec:pants}
Now we introduce the notation that will be used in the definition of the product. This model was considered first in \cite[Section 3.2]{AS-gt}. Let $(P,j)$ be the Riemannian surface obtained by removing three points, $a_1$, $a_2$, $b$, from the sphere. Let $S\subset P$ be the subset of $P$ consisting of two intersecting curves as in Figure~\ref{fig_pp} such that the complement of $S$ is biholomorphic to $\R\times(0,1)\sqcup \R\times(0,1)$. 
\begin{figure}[h]
\includegraphics[scale=0.9]{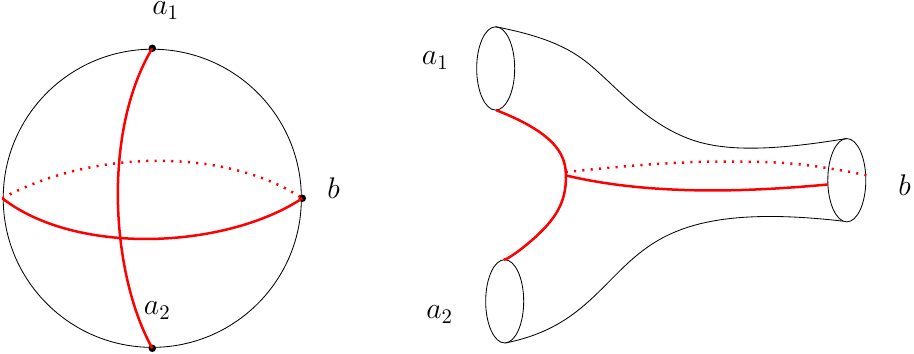}
\centering
\caption{Pair-of-pants with a slit.} \label{fig_pp}
\end{figure}
Let $\psi_1, \psi_2: \R\times(0,1)\to P\setminus S$ be restrictions of a biholomorphism $\R\times(0,1)\sqcup \R\times(0,1)\to P\setminus S$ to the connected components of $\R\times(0,1)\sqcup \R\times(0,1)$. Assume, by modifying the biholomorphism if necessary,
\begin{align*}
& \lim_{s\to -\infty} \psi_1(s,t)=a_1, \quad  \lim_{s\to +\infty} \psi_1(s,t)=b,\\
& \lim_{s\to -\infty} \psi_2(s,t)=a_2, \quad  \lim_{s\to +\infty} \psi_2(s,t)=b.
\end{align*} 
The limits above are limits in the sphere. Let $\delta>0$ be a sufficiently small positive number (say, smaller than $\frac{1}{2}$) and let $\Omega=\Omega_\delta\subset P$ be the open neighbourhood of $S$ given by
\[\Omega:= \left(P\setminus \psi_1( \R\times [\delta, 1-\delta] )\right)\setminus \psi_2( \R\times [\delta, 1-\delta] ).\]
Let 
\begin{align*}
&\iota_{a_1}, \iota_{a_2}: (-\infty, 0]\times\mathbb{S}^1\to P,\\
& \iota_b: [0,+\infty)\times\mathbb{S}^1\to P
\end{align*}
be biholomorphic embeddings such that 
\begin{align*}
& \iota_{a_1}(s,t)= \psi_1(s-s_a, t)\quad \text{for } s \in (-\infty, 0] \text{ and } t \in (0, 1),\\
& \iota_{a_2}(s,t)= \psi_2(s-s_a, t)\quad \text{for } s \in (-\infty, 0] \text{ and } t \in (0, 1),\\
& \iota_{b}\left(\frac{s}{2},\frac{t}{2}\right)= \psi_1\left(s+s_b, t\right)\quad \text{for } s \in[0,+\infty) \text{ and } t \in (0, 1),\\
& \iota_{b}\left(\frac{s}{2},\frac{t+1}{2}\right)= \psi_2\left(s+s_b, t\right)\quad \text{for } s \in[0,+\infty) \text{ and } t \in (0, 1),
\end{align*}
for positive $s_a, s_b\in\R^+$. Let $\beta$ be a $1-$form on $P$ such that $\psi^{*}_1 \beta = dt$ on $\mathbb{R} \times [\delta, 1-\delta]$ and $\psi^{*}_2 \beta = dt$ on $\mathbb{R} \times [\delta, 1-\delta]$.

\subsection{Product data}\label{sec:productdata}
Now, we define \emph{product data}. Let $\mathcal{D}^{a_1}:=(H^{a_1}, J^{a_1})$, $\mathcal{D}^{a_2}:=(H^{a_2}, J^{a_2})$, and $\mathcal{D}^b:=(H^{b}, J^{b})$ be regular Floer data such that the Hamiltonians $H^{a_1}_t$ and $H^{a_2}_t$ are constant for $t\in [-\delta, \delta] $, and such that $H^b_t$ is constant for $t\in \left[-\frac{\delta}{2}, \frac{\delta}{2}\right]\cup\left[\frac{1-\delta}{2}, \frac{1+\delta}{2}\right]$. The product data for the triple $(\mathcal{D}^{a_1}, \mathcal{D}^{a_2};\mathcal{D}^{b})$ consists of a (smooth) $P$-depended Hamiltonian
\[H:\hat{W}\times P\to \R : (x,p)\mapsto H_p(x)\]
and of a (smooth) $P$-family $J_p$ of $d\lambda$-compatible almost complex structures on $\hat{W}$ such that the following conditions hold.
\begin{enumerate}
\item (Conditions on the ends of $P$).
\begin{align*}
& H(x, \iota_{a_1}(s,t))=H^{a_1}_t(x)\quad\text{for } s\leqslant -1,\\
& H(x, \iota_{a_2}(s,t))=H^{a_2}_t(x)\quad\text{for } s\leqslant -1,\\
& H(x, \iota_{b}(s,t))=H^{b}_t(x)\quad\text{for } s\geqslant 1,\\
&J_p=\left\{ \begin{matrix} J^{a_1}_t& \text{if } p\in\iota_{a_1}((-\infty,-1]\times\mathbb{S}^1), \\ J^{a_2}_t & \text{if } \iota_{a_2}((-\infty,-1]\times\mathbb{S}^1),\\ J^b_t & \text{if } \iota_{b}([1,+\infty)\times\mathbb{S}^1).\end{matrix}\right.
\end{align*}
\item (The Hamiltonian on the conical end of $\hat{W}$). There exist $T\in\R^+$ and a smooth function $h: \partial W\to\R$ such that 
\[H_p(y,r)= r\cdot h(y,p) \]
 for $r\geqslant T$, $y\in\partial W$, and $p\in P$.
\item (The almost complex structure on the conical end of $\hat{W}$).  There exists $T\in\R^+$ such that $J_p$ is of twisted SFT-type on $\partial W\times[T,+\infty)$ for all $p\in P$.
\item (Conditions around the slit). For all $p\in\Omega$, the contact Hamiltonian $h_p:= h(\cdot, p)$ is equal to 0.
\item (Monotonicity). The functions 
\[s\mapsto h(y, \psi_1(s,t)),\quad s\mapsto h(y, \psi_2(s,t))\]
are increasing for all $t\in (0,1)$ and $y\in \partial W$.
\end{enumerate}

The definition of product data uses the notion of almost complex structure of twisted SFT-type. We now recall this notion. Let $(M, \xi)$ be a contact manifold with a contact form $\alpha$. Let $q:M\to \R^+$ be a smooth positive function. Denote by $Q:M\times \R^+\to \R^+$ the corresponding homogeneous Hamiltonian on the symplectization, i.e. $Q(y, r)= r\cdot q(y)$. Let $N_Q$ be the following distribution in $TM$
\[ N_Q(y,r):= \left\{\left(v, -\frac{r dq(y) v}{q(y)}\right)\:|\: v\in\xi_{y} \right\}. \]
Note that $dQ$ vanishes on $N_Q$. For $v\in \xi_y$, denote 
\[\zeta_Q(v)^r:= \left(v, -\frac{r dq(y) v}{q(y)}\right)\in N_Q(y,r). \]
An almost complex structure $J$ on $M\times\R^+$ if said to be of twisted SFT-type (with the twist $Q$) is there exists a $d\alpha$-compatible complex structure $j$ on $\xi$ such that the following holds
\begin{itemize}
\item $ J X_Q= r\partial_r, $
\item $ J^r\zeta_Q(v)= \zeta^r_Q(jv). $
\end{itemize}
If $J$ is of twisted SFT-type with the twist $Q$, then $dQ\circ J=-\lambda$.

\subsection{Definition of the pair-of-pants product}\label{sec:defPPP}

Let $h^{a_1}, h^{a_2},$ and $h^{b}$ be admissible contact Hamiltonians as at the beginning of Section \ref{sec-pp}. Let $\mathcal{D}^{a_1}:=(H^{a_1}, J^{a_1})$, $\mathcal{D}^{a_2}:=(H^{a_2}, J^{a_2})$, and $\mathcal{D}^b:=(H^{b}, J^{b})$ be regular Floer data such that the Hamiltonians $H^{a_1}, H^{a_2}, H^b$ have slopes equal to $h^{a_1}, h^{a_2}, h^b,$ respectively, and such that they satisfy the following condition: the Hamiltonians $H^{a_1}_t$ and $H^{a_2}_t$ are constant for $t\in [-\delta, \delta] $, and the Hamiltonian $H^b_t$ is constant for $t\in \left[-\frac{\delta}{2}, \frac{\delta}{2}\right]\cup\left[\frac{1-\delta}{2}, \frac{1+\delta}{2}\right]$ for sufficiently small $\delta>0$. Let $(H,J)$ be product data for the triple $(\mathcal{D}^{a_1}, \mathcal{D}^{a_2};\mathcal{D}^{b})$.
The pair-of-pants product
\[\Pi: {\rm HF}_\ast(h^{a_1})\otimes {\rm HF}_\ast(h^{a_2})\to {\rm HF}_\ast(h^b)\]
is defined, on the chain level, on generators, by
\[\Pi(\gamma_1\otimes \gamma_2):= \sum_{\gamma_3} n(\gamma_1, \gamma_2;\gamma_3) \left<\gamma_3\right>,\]
where $n(\gamma_1,\gamma_2;\gamma_3)$ is the number modulo 2 of the isolated solutions $u:P\to\hat{W}$ of the problem
\begin{align*}
    &(du-X_H\otimes\beta)^{0,1}=0,\\
    & \lim_{s\to-\infty} u\circ\iota_{a_1}(s,t)=\gamma_1(t),\\
    & \lim_{s\to-\infty} u\circ\iota_{a_2}(s,t)=\gamma_2(t),\\
    & \lim_{s\to+\infty} u\circ\iota_{b}(s,t)=\gamma_3(t).
\end{align*}

\subsection{Compactness of the moduli space}\label{sec:compactnessPPP}

Now, we prove that the elements of the moduli space cannot ``escape'' to the conical end of $\hat{W}$. Namely, we prove the following result, where we list it as a theorem due to the potential interest by its own. 

\begin{theorem}\label{thm:noescape}
    Let $P, j, \beta$ be as in Section~\ref{sec:pants}. Let $E\in\R^+$ and let $(H,J)$ be product data on the completion $\hat{W}$ of a Liouville domain. There exists a compact $K = K_{H,J,E}\subset \hat{W}$ such that $u(P)\subset K$ for every solution $u: P\to\hat{W}$ of the Floer equation
    \[\left(du - X_H(u)\otimes \beta\right)^{0,1}=0\]
    that satisfies $\mathbb{E}(u)=\frac{1}{2}\int_P  \abs{du - X_{H_p}\otimes\beta}^2  d{\rm vol}_P< E.$
\end{theorem}
To this end, we use the Aleksandrov maximum principle (see \cite[Theorem 9.1]{GT01}, \cite[Appendix A]{AS09}, and  \cite{MU}). The following proposition is a coordinate-free reformulation of the Aleksandrov maximum principle. 

\begin{prop}\label{prop:aleksandrov}
    Let $(\Sigma, j)$ be an open planar Riemannian surface (i.e. $(\Sigma, j)$ is a planar Riemannian surface that is not the sphere) with a volume form $dvol_\Sigma$. Let $K\subset \Sigma$ be a compact subset and let $A\in\R^+$. Then, there exists a constant $C_{K,A}\in\R^+$ such that for every connected open subset $\Omega\subset \Sigma$ with $\Omega\subset K$ and every 1-form $\eta$ on $\Omega$ with $\norm{\eta}_{L^2}\leqslant A$, the following holds. If $\mu, f:\Omega\to \R$ are smooth functions such that
    \begin{equation}\label{eq:ineqddc} -dd^c\mu + \eta\wedge d\mu \geqslant  f d{\rm vol}_\Sigma,\end{equation}
    then
    \[ \sup_{\Omega}\mu \leqslant \sup_{\partial \Omega} \mu^+ + C_{K,A}\cdot \int_\Omega f^2 d{\rm vol}_\Sigma.\]
    Here, $\mu^+:=\max(\mu, 0)$.
\end{prop}

By definition, for every product data $(H, J)$ there exists $T=T_{H,J}\in\R^+$ such that $H(y,r,p)=r\cdot h(y,p)$ for $(y,r)\in (\partial W)\times[T, +\infty)$ and for some ($P$-parametrized) contact Hamiltonian $h:(\partial W)\times P\to\R$, and such that the restriction of $J_p$ to $(\partial W)\times[T, +\infty)$ is of twisted SFT-type. We denote by $Q^J:(\partial W)\times\R^+\times P\to \R^+$ the function such that $Q^J(\cdot,\cdot, p)$ is the twist of $J_p$. The next proposition shows that the function $p\mapsto \log(Q^J(u(p),p))$ satisfies the inequality of the type \eqref{eq:ineqddc}.

\begin{prop}\label{prop:ineq}
Let $(H,J)$ be product data on $\hat{W}$. Let $T=T_{H,J}$ and $Q=Q^J$. Let $u:P\to \hat{W}$ be a solution of the Floer equation
$0= \left( du - X_H(u)\otimes\beta\right)^{0,1}.$
Let $\Omega:=u^{-1}(\partial W\times(T, +\infty)).$ Then, the function $\mu:\Omega\to \R : p\mapsto \log Q(u(p), p)$ satisfies the inequality
\[ -dd^C\mu + \theta\wedge d\mu + d\theta\geqslant 0, \]
where
\[\theta:=\frac{(d_PQ\circ j)(u) + \{Q, H\}(u)\beta\circ j}{Q\circ u}.\]
\end{prop}
\begin{proof}
Let $\rho:\Omega\to \R^+$ be the function given by $\rho(p):=Q_p(u(p))= e^{\mu(p)}$. We start the proof by computing $dd^C\rho=d(d\rho\circ j)$. The Floer equation, together with $d_WQ\circ J=-\lambda$, implies
\begin{align*}
d\rho = & d (Q\circ u)\\
	= & (d_P Q)\circ u + (d_W Q)(du)\\ 
	= & (d_P Q)\circ u + (d_W Q)(X_{H}(u)\beta - J(u)du\circ j + J(u)X_H(u)\beta\circ j)\\ 
	= &  (d_P Q)\circ u + \{X_Q, X_H\}(u)\beta + (d_W Q)J(u)( - du\circ j + X_H(u)\beta\circ j )\\
	= &  (d_P Q)\circ u + \{X_Q, X_H\}(u)\beta - \lambda ( - du\circ j + X_H(u)\beta\circ j )\\
	= &  (d_P Q)\circ u + \{X_Q, X_H\}(u)\beta + \lambda ( du\circ j ) - \lambda(X_H)(u)\beta\circ j.
\end{align*}
Therefore,
\begin{align*}
d\rho\circ j = & (d_P Q\circ j)\circ u + \{X_Q, X_H\}(u)\beta\circ j - \lambda ( du )+ \lambda(X_H)(u)\beta\\
		= &  (d_P Q\circ j)\circ u + \{X_Q, X_H\}(u)\beta\circ j - u^\ast \lambda+ \lambda(X_H)(u)\beta,
\end{align*}
and consequently,
\begin{align*}
dd^C \rho = & -u^\ast\omega + d (\lambda(X_H)(u))\wedge\beta +\lambda(X_H(u)) d\beta + d(\rho\theta)\\
		= & -u^\ast\omega - d (H(u))\wedge\beta -H(u) d\beta + d\rho\wedge\theta+ \rho d\theta.
\end{align*}
In the last equation, we also used $\lambda(X_H)=-H$. Since
\[ u^\ast \omega + d_WH(du)\wedge\beta = \frac{1}{2}\norm{du- X_H\otimes\beta}^2d{\rm vol}_P, \]
we have
\begin{align*}  
dd^C\rho =& -\frac{1}{2}\norm{du - X_H(u)\otimes\beta}^2d{\rm vol}_P + d\rho\wedge\theta + \rho d\theta - (d_P H)\circ u \wedge\beta  - H(u)d\beta.
\end{align*}
The energy density $\norm{du - X_H(u)\otimes\beta}^2dvol_P$ of the (vector-bundle-valued) 1-form $du - X_H(u)\otimes\beta$ can be estimated from below by the energy density of the projection of $du - X_H(u)\otimes\beta$ to $\{k\cdot X_Q(u) + \ell\cdot \partial_r\:|\:k, \ell\in\R\}$. Now, we compute the projection to $\{k\cdot X_Q(u)\:|\:k\in\R\}$:
\begin{align*}
\frac{\omega( du - X_H(u)\otimes\beta, J X_Q(u) )}{\norm{X_Q(u)}_J} &= \frac{\omega(J du - JX_H(u)\otimes\beta, - X_Q(u))}{\sqrt{\rho}}\\
							&= dQ(J du - JX_H(u)\otimes\beta)\cdot\rho^{-\frac{1}{2}}\\
							&= dQ( du\circ j - X_H\otimes\beta\circ j)\cdot\rho^{-\frac{1}{2}}\\
							&= \big(d(Q(u))\circ j - (d_PQ\circ j)(u)\\
							&\phantom{\&} - \{Q, H\}(u)\beta\circ j\big)\cdot\rho^{-\frac{1}{2}} = (d\rho\circ j - \rho\theta)\cdot\rho^{-\frac{1}{2}}.
\end{align*}
Similarly, the projection to $\{k\cdot \partial_r\:|\:k\in\R\}$ is equal to
\begin{align*}
    \frac{\omega( du - X_H(u)\otimes\beta, J \partial_r )}{\norm{\partial_r}_J} &=  \frac{\omega( du - X_H(u)\otimes\beta, J \rho\partial_r )}{\norm{\rho\partial_r}_J}\\
    &= \rho^{-\frac{1}{2}}\cdot \omega\left( du- X_H(u)\otimes\beta, -X_Q(u) \right)\\
    &= (d\rho + \rho\cdot \theta\circ j)\cdot \rho^{-\frac{1}{2}}.
\end{align*}
Using the considerations above and the orthogonality of $X_Q$ and $\partial_r$, we get
\begin{align*}
\norm{du - X_H(u)\otimes\beta}^2d{\rm vol}_P =& \norm{(d\rho\circ j - \rho\theta)\cdot\rho^{-\frac{1}{2}}}^2 + \norm{(d\rho + \rho\cdot \theta\circ j)\cdot \rho^{-\frac{1}{2}}}^2\\
            =& 2\cdot  \norm{(d\rho\circ j - \rho\theta)\cdot\rho^{-\frac{1}{2}}}^2 \\
            \geqslant & \frac{2}{\rho}\cdot (d\rho\circ j - \rho\theta)\circ j\wedge (d\rho\circ j - \rho\theta)\\
		=& \frac{2}{\rho}\cdot\left( -d\rho -\rho\theta\circ j \right)\wedge \left( d\rho\circ j -\rho\theta \right) \\
		=& \frac{2}{\rho}\cdot \big(-d\rho\wedge d\rho\circ j + \rho d\rho \wedge\theta\\
			&  - \rho (\theta\circ j)\wedge(d\rho\circ j) + \rho^2\cdot(\theta\circ j)\wedge\theta\big) \\
		=&\frac{2}{\rho}\cdot ( (d^C\rho)\wedge d\rho + 2 \rho d\rho\wedge\theta + \rho^2\cdot(\theta\circ j)\wedge\theta )\\
		\geqslant& \frac{2}{\rho}\cdot (d^C\rho)\wedge(d\rho) + 4 d\rho\wedge\theta.
\end{align*}
Therefore,
\begin{align*}
-dd^C\rho\geqslant& \frac{1}{\rho}\cdot d^C\rho\wedge d\rho + d\rho\wedge\theta - \rho d\theta + (d_PH)(u)\wedge\beta + H(u)d\beta.
\end{align*}
Since 
\[ d\mu=\frac{d\rho}{\rho},\quad dd^C\mu=\frac{dd^C\rho}{\rho} +\frac{d^C\rho\wedge d\rho}{\rho^2}, \]
we have
\[-dd^C\mu + \theta\wedge d\mu + d\theta \geqslant \frac{(d_PH)(u)\wedge\beta + H(u)d\beta}{Q(u)}. \]
This finishes the proof.
\end{proof}

The following lemma shows that there exists a real number $T>1$ such that the preimage of $\partial W\times [T, +\infty])$ under every $u$ as in Theorem~\ref{thm:noescape} has uniformly bounded diameter.
\begin{lemma}
 Let $(H,J)$ be product data on $\hat{W}$ and let $E\in\R^+$ be a positive real number. Then, there exist real numbers $T>1, L>0,$ and a compact subset $K\subset P$ such that the following holds. For every solution $u:P\to\hat{W}$ of the Floer equation $(du - X^H(u)\otimes\beta)^{0,1}=0$ with $\mathbb{E}(u)<E$ and every connected component $\Omega$ of $u^{-1}(\partial W\times [T,+\infty))$ at least  one of the following conditions holds:
 \begin{enumerate}
     \item $\Omega\subset K$,
     \item There exists an interval $I$ of length at most $L$ such that $\Omega\subset \iota_{c}(I\times\mathbb{S}^1)$ for some $c\in\{a_1, a_2, b\}.$
 \end{enumerate}
\end{lemma}
\begin{proof}
By Lemma~\ref{lem:gammainnbhd}, there exist positive real numbers $\varepsilon_c, B_c\in\R^+,$ $c\in\{a_1,a_2, b\}$ such that $\gamma(\mathbb{S}^1)\subset \hat{W}\setminus (\partial W\times(B_c, +\infty))$ for every loop $\gamma:\mathbb{S}^1\to \hat{W}$ that satisfies
\[\int_0^1\norm{ \gamma'(t)- X_{H^c_t}(\gamma(t)) }_{J_c}^2 dt\leqslant \varepsilon_c.\]
Denote
\[T := \max_c B_c,\quad \varepsilon := \min_c \varepsilon_c,\quad L:= \frac{E}{\varepsilon} + 1. \] Let $K\subset P$ be the complement of 
\[\iota_{a_1}\left((-\infty, -L)\times \mathbb{S}^1\right)\:\cup\: \iota_{a_2}\left((-\infty, -L)\times \mathbb{S}^1\right)\:\cup\: \iota_{b}\left((L,+\infty)\times \mathbb{S}^1\right). \]
Now, we check that $T, L,$ and $K$ satisfy the conditions of the lemma. First, we show that for every connected component $\Omega'$ of $\Omega\cap \op{im} \iota_c$, where $c\in\{a_1, a_2, b\}$, there exists an interval $I$ of length at most $L$ such that $\Omega'\subset \iota_c\left( I\times\mathbb{S}^1 \right)$. Assume the contrary. Then, there exists an interval $[s_0, s_0+ L]$ such that $\iota_c \left( \{s\}\times\mathbb{S}^1 \right) \cap \Omega'\not=\emptyset$ for all $s\in[s_0, s_0+L]$. This implies
\[\int_0^1\abs{ \partial_t(u\circ\iota_c)(s,t)- X_{H^c_t}(u\circ\iota_c(s,t)) }_{J_c}^2 dt> \varepsilon_c\]
for all $s\in[s_0, s_0+L]$. Hence, if $c=b$,
\begin{align*}
    E & \geqslant \frac{1}{2} \int_P \abs{du - X_{H_p}\otimes\beta}^2 d{\rm vol}_P\\
    &\geqslant \frac{1}{2} \int_{\iota_c([s_0, s_0+L]\times\mathbb{S}^1)} \abs{du - X_{H_p}\otimes\beta}^2 d{\rm vol}_P\\
    & = \int_{s_0}^{s_0+L}\int_0^1 \abs{ \partial_t (u\circ\iota_c(s,t)) - X_{H^c_t}(u\circ\iota_c(s,t))}^2 dtds\\
    & > \int_{s_0+1}^{s_0+L}\int_0^1 \abs{ \partial_t (u\circ\iota_c(s,t)) - X_{H^c_t}(u\circ\iota_c(s,t))}^2 dtds\\
    &> \int_{s_0+1}^{s_0+L} \varepsilon_c ds\\
    &= (L-1)\cdot \varepsilon_c\geqslant E,
\end{align*}
which is a contradiction. If $c\in\{a_1, a_2\}$, one similarly obtains a contradiction. If $\Omega\subset \bigcup_c \im{\iota_c}$, then the argument above implies that there exist $c\in\{a_1, a_2, b\}$ and an interval $I$ of length at most $L$ such that $\Omega\subset \iota_c(I\times\mathbb{S}^1)$. Otherwise, $\Omega\subset K$. Indeed, if $\Omega\not\subset K$ and if $\Omega\not \subset \bigcup_c\op{im}\iota_c$, then for some $c\in\{a_1, a_2, b\}$ and for some connected component $\Omega'$ of $\Omega\cap \op{im}\iota_c$ the sets $\iota_c \left( \{0\}\times\mathbb{S}^1 \right)\cap \Omega'$ and $\iota_c \left( \{(\op{sgn} c) \cdot L \}\times\mathbb{S}^1 \right)\cap \Omega'$ are non-empty (where $\op{sgn} c = -1$ if $c \in \{a_1, a_2\}$ and $\op{sgn} c = 1$ if $c = b$). This is a contradiction by previous considerations. 
\end{proof}

Now, we formulate the statement that was used in the proof of the last lemma and give a reference for the proof.
\begin{lemma}\label{lem:gammainnbhd}
    Let $(H, J)$ be a regular Floer data. Then, there exist $B, \varepsilon\in\R^+$ such that for every smooth loop $\gamma: \mathbb{S}^1\to \hat{W}$ the following holds. If 
    \[\int_0^1\norm{ \gamma'(t)- X_{H_t}(\gamma(t)) }_J^2 dt\leqslant \varepsilon,\]
    then $\gamma(\mathbb{S}^1)\subset \hat{W}\setminus \left(\partial W\times(B, +\infty)\right).$
\end{lemma}
\begin{proof}
    See Lemma 3.6 in \cite{MU}.
\end{proof}

\begin{proof}[Proof of Theorem~\ref{thm:noescape}]
    Let $u:P\to\hat{W}$ denote a solution of the Floer equation $\left(du - X_H(u)\otimes \beta\right)^{0,1}=0$ that satisfies $\mathbb{E}(u)<E$. Let $T=T_{H,J}$ and $Q=Q^J$ be as in Proposition~\ref{prop:ineq}. Denote $\Omega:= u^{-1}(\partial W\times(T, +\infty)).$ By Proposition~\ref{prop:ineq}, the function $\mu:\Omega\to\R : p\mapsto \log Q(u(p), p)$ satisfies the inequality
    \[ -dd^C\mu + \theta\wedge d\mu + d\theta\geqslant 0, \]
    where
    \[\theta:=\frac{(d_PQ\circ j)(u) + \{Q, H\}(u)\beta\circ j}{Q\circ u}.\]
    In the view of Proposition~\ref{prop:aleksandrov} (the Alexandrov weak maximum principle), it is enough to show that $\norm{\theta}_{L^2}$ and $\norm{d\theta}_{L^2}$ are bounded by a constant that does not depend on $u$. Denote
    \[(w, r):= \left. u \right|_{\Omega} : \Omega\to \partial W\times (T, +\infty).\]
    Since $\partial W$ is compact and since
    \[\theta =\frac{(d_Pq\circ j)(w)}{q(w)} + dh(R),\]
    it is enough to check that $\norm{dw}_{L^2}$ is bounded by a constant that does not depend on $u$. Here, $q$ and $h$ are the slopes of $Q$ and $H$, respectively. Let $g$ be a Riemannian metric on $\partial W$. Then, there exists $\varepsilon>0$ such that 
    \[\norm{\zeta + a r\partial_r}^2_J\geqslant \varepsilon r\cdot \norm{\zeta}_g^2\]
    for all $a\in \R$ and $\zeta\in T(\partial W)$. Hence,
    \begin{equation*}
        \norm{d w}_g^2 \leqslant \frac{1}{\varepsilon r}\cdot \norm{d u}_J^2\leqslant \frac{2}{\varepsilon r}\cdot\left( \norm{du - X_H(u)\otimes\beta}^2_J + \norm{X_H(u)\otimes\beta}^2_J \right).
    \end{equation*}
Since $\frac{1}{r}\norm{X_H(p,r)}_J$ is bounded and 
    \[\int_\Omega \norm{du - X_H(u)\otimes\beta}^2_J d{\rm vol}_P<E, \]
we finish the proof. 
\end{proof}

\section{Proofs of the main results} \label{sec-proof-main}

Let $(M, \xi)$ be a contact manifold with a Liouville filling $(W, \lambda)$. Recall that for any contact Hamiltonian $h: [0,1] \times M \to \R$, the ${\rm osc}(h)$-gapped module $\mathbb V_h$ is defined in Section \ref{sec-gap-mod} and briefly denoted by $\{{\rm HF}_*(\eta \#h)\}_{\eta \in \R \backslash \mathcal S_h}$, where $\mathcal S_h$ is defined in (\ref{dfn-RFspec}). By Definition \ref{dfn-si-gap}, for any class $a \in {\rm SH}_*(W)$, one defines the contact spectral invariant $c(a, \mathbb V_h)$. In this section, we will prove all those properties of $c(a, \mathbb V_h)$ as listed in Section \ref{sec-mra}, except the proof of Theorem \ref{thm-descend} will be postponed to Section \ref{sec-more-wd}. 

\begin{proof} [Proofs of Theorem \ref{intro-thm-1}] For admissible contact Hamiltonians $h, g: [0,1] \times M \to \R$, consider ${\rm osc}(h)$-gapped module $\V_h$ and ${\rm osc}(g)$-gapped module $\V_g$, 
\[ \V_h = \{{\rm HF}_*(\eta \#h)\}_{\eta \in \R \backslash \mathcal S_h} \,\,\,\,\mbox{and}\,\,\,\, \V_g = \{{\rm HF}_*(\eta \#g)\}_{\eta \in \R \backslash \mathcal S_g} \]
as constructed and discussed in Section \ref{sec-construction} and Example \ref{ex-chfh-gap}. For any almost optimal $m_{h,g}$-gapped sequence $\mathfrak a = \{\cdots, \eta_0, \eta_1, \cdots\}$ (which always exists, since the union of spectra in (\ref{dfn-RFspec}), $\mathcal S_h \cup \mathcal S_g$, is a discrete subset of $\R$), by definition, we have $\eta_{i+1} = \eta_i + m_{g,h} + {\o}^{\mathfrak a}_{i}(1)$ and one can study the corresponding $m_{h,g}$-gapped restrictions $\mathbb V_h(\mathfrak a)$ and $\mathbb V_g(\mathfrak a)$. If $h \neq g$, then we have the following computation, 
\begin{align*}
(\eta_{i+1} \# g) - (\eta_{i}\#h) & = (\eta_{i+1} + g_t \circ \phi_R^{-t \eta_{i+1}}) - (\eta_{i} + h_t \circ \phi_R^{-t \eta_{i}})
\\
& = (\eta_{i+1} -\eta_{i}) -  (h_t \circ \phi_R^{-t\eta_{i+1}} - g_t \circ \phi_R^{-t\eta_{i}}) \\
& \geq  \max\{{\rm osc}(g,h), {\rm osc}(h, g)\} - \left(\int_0^1 \max_M h_t - \min_M g_t\, dt\right) \geq 0.
\end{align*}
Therefore, $\eta_{i}\#h\leq \eta_{i+1} \# g$. A symmetric argument shows that $\eta_{i}\#g\leq \eta_{i+1} \# h$. Then we have the following commutative diagram, 
\[ \xymatrix{
{\rm HF}_*(\eta_{i}\#h) \ar[rr]^-{\iota^{\mathbb V_h(\mathfrak a)}_{\eta_{i} \, \eta_{i+1}}} \ar[rd]_-{\phi_i} && {\rm HF}_*(\eta_{i+1} \# h)\\
& {\rm HF}_*(\eta_{i}\#g)\ar[ru]_-{\psi_{i+1}}} \]
as well as a symmetric one. Here, $\phi_i$ and $\psi_i$ are Floer continuation maps. If $h = g$, then ${\rm osc}(h,g) = {\rm osc}(h)$ and obviously $\mathbb V_h$ and $\mathbb V_g$ are $0$-interleaved. These confirm the two diagrams in Definition \ref{dfn-gap-inter}, where $\lambda = {\rm osc}(h), \lambda' = {\rm osc}(g)$, and set $\delta = m_{h,g}$ if $h \neq g$ and $\delta = 0$ if $h = g$. \end{proof}

\begin{proof} [Proof of Theorem \ref{prop-shift}] 
By the definition of the operation $\#$, we have 
\begin{align*}
\eta \#(s \#h) & = \eta \# (h \circ \phi_R^{-st} + s) = h \circ \phi_R^{-(s+\eta) t} + (s+\eta). 
\end{align*}
This means that the parameterization set $\R \backslash \mathcal S_{s \# h}$ is a shift of $\R \backslash \mathcal S_h$ by $+s$, that is, $\R \backslash \mathcal S_h + s$.  Up to $\ep>0$, suppose $c(a, \mathbb V_h)$ is obtained via an ${\rm osc}(h)$-normalized persistence module $\mathbb V_h(\mathfrak a)$ for an index sequence $\mathfrak a: \Z \to \R \backslash \mathcal S_h$. Consider a new index sequence $\mathfrak b: \Z \to \R$ defined by 
\[ \mathfrak b(i) : = \mathfrak a(i) + s, \]
which results in an ${\rm osc}(s \#h)$-normalized persistence module $\mathbb V_h(\mathfrak b)$. Moreover, by the standard shifting property of the spectral invariant (of class $a \in {\rm SH}_*(W)$), we have 
\[ c(a, \mathbb V_{s \# h}(\mathfrak b)) = c(a, \mathbb V_h(\mathfrak a)) + s = c(a, \mathbb V_h) + s. \]
Therefore, by definition, $c(a, \mathbb V_{s \#h}) \leq c(a, \mathbb V_h) + s$. Then a symmetric argument implies the desired conclusion. 
\end{proof}

Here, let us emphasize a slightly different setting, where a similar argument as in the proof of Proposition \ref{prop-shift} would lead to a completely different conclusion. Explicitly, the discussion below shows that the shift property in Theorem \ref{prop-shift} essentially comes from the shift of the ambient index set $\R \backslash \mathcal S_h$, instead of the shift on the gapped restrictions (as persistence modules). 

Assume the index set $\R \backslash \mathcal S_h$ is {\it not} shifted. Still, let us consider the index sequence $\mathfrak b: \Z \to \R$ as above defined by $\mathfrak b(i) : = \mathfrak a(i) + s$. Without loss of generality, assume $s \geq 0$. By adjusting the constants ${\o}^{\mathfrak a}_i(1)$ in (1) of Definition \ref{dfn-restriction}, together with the discreteness of $\mathcal S_h$ given by Lemma \ref{lemma-spec}, we can assume that $\mathfrak b: \Z \to \R \backslash \mathcal S_h$ as well. Then $\mathbb V_h(\mathfrak b)$ is an ${\rm osc}(h)$-almost optimal persistence module (but not necessarily ${\rm osc}(h)$-normalized). Then by Remark \ref{rmk-generalized-csi}, for any class $a \in {\rm SH}_*(W)$, we have 
\[ c(a, \mathbb V_{s\#h}) \leq c(a, \mathbb V_h(\mathfrak b)) - m^{\mathfrak b} {\rm osc}(h) \]
where $m^{\mathfrak b}$ is the unique integer such that $\mathfrak b - m^{\mathfrak b} {\rm osc}(h)$ is ${\rm osc}(h)$-normalized. Here, more explicitly, $m^{\mathfrak b} =  \floor{\frac{s}{{\rm osc}(h)}}$. Moreover,  we have 
\begin{align*}
c(a, \mathbb V_{s\#h}) & \leq  c(a, \mathbb V_h(\mathfrak a)) + s -  \left\lfloor{\frac{s}{{\rm osc}(h)}}\right\rfloor \cdot {\rm osc}(h) \\
& \leq c(a, \mathbb V_h(\mathfrak a)) + s - \left(\frac{s}{{\rm osc}(h)} -1 \right)\cdot {\rm osc}(h)\\
& = c(a, \mathbb V_h(\mathfrak a)) + s - s + {\rm osc}(h) = c(a, \mathbb V_h) + {\rm osc}(h)
\end{align*}
where we use the fact that $c(a, \mathbb V_h(\mathfrak b)) = c(a, \mathbb V_h(\mathfrak a)) + s$. By a symmetric argument that replaces $s$ by $-s$, we get the following estimation 
\[ |c(a, \mathbb V_{s\# h}) - c(a, \mathbb V_h)| \leq {\rm osc}(h)\]
for any $s \in \R$. In fact, if $\frac{s}{{\rm osc}(h)} \in \Z$, then $c(a, \mathbb V_{s\# h}) = c(a, \mathbb V_h)$. 

\begin{proof} [Proof of Theorem \ref{prop-spectrality}] Since the set $\mathcal S_h$ is discrete, if $c(a,\mathbb V_h) \notin \mathcal S_h$, then there exists some sufficiently small $\ep>0$ such that 
\[ (c(a,\mathbb V_h) - \ep, c(a,\mathbb V_h) + \ep) \subset \R \backslash \mathcal S_h.\] 
By definition, for this $\ep$, there exists an ${\rm osc}(h)$-gapped normalized restriction $\mathbb V(\mathfrak a)$ for some $\mathfrak a: \Z \to \R \backslash \mathcal S_h$ such that $c(a,\mathbb V_h) \leq c(a, \mathbb V(\mathfrak a)) < c(a,\mathbb V_h) + \ep$. Moreover, we have $c(a, \mathbb V(\mathfrak a)) = \eta_{i_*} = \mathfrak a(i_*)$ for some $i_* \in \Z$, due to the spectrality of persistence module $c(a, \mathbb V(\mathfrak a))$ from \cite{Ush-spec}. By adjusting the constant ${\o}^{\mathfrak a}_i(1)$ in (1) in Definition \ref{dfn-restriction}, we can assume that the distance from $\eta_i$ to $\mathcal S_h$ is at least $\ep$ for {\it any} $i \in \N$. In particular, this holds for $\eta_{i_*}$ by its defining property, without changing the associated ${\o}^{\mathfrak a}_{i_*}(1)$.

Now, consider a new $s$-family of index sequences $\mathfrak b_s: \Z \to \R$ by $$\mathfrak b_s(i) := \mathfrak a(i) - s \ep = \eta_i  - s \ep$$
for $s \in [0,1]$. We can shrink $\ep$ if necessary so that the following $s$-family contact Hamiltonians
\[ (\eta_i - s \ep) \# h = (-s \ep) \# (\eta_i \#h) \]
are all admissible for any $s \in [0,1]$ and for any $i \in \Z$ (since $\eta_i \#h$ are assumed to be admissible). Then \cite[Theorem 1]{UZ22} implies that ${\rm HF}_*({\eta}_i \#h) \simeq {\rm HF}_*((\eta_i - s\ep) \#h)$. Moreover, for sufficiently small $\ep$, the index sequences $\mathfrak b_1$ is also normalized. Therefore, consider the ${\rm osc}(h)$-gapped restriction $\mathbb V(\mathfrak b_1)$. The argument above implies that $\mathbb V(\mathfrak b_1)$ is simply a $(-\ep)$-shift of $\mathbb V(\mathfrak a)$. Therefore, by definition, we have 
\[ c(a, \mathbb V_h) \leq c(a, \mathbb V(\mathfrak b_1)) = c(a, \mathbb V(\mathfrak a)) - \ep < c(a, \mathbb V_h) \]
which provides the desired contradiction. \end{proof}

\begin{remark} The computational Example in Section \ref{sec-ex} supports Theorem \ref{prop-spectrality}. \end{remark}

\begin{proof} [Proof of Theorem \ref{prop-duality}] Up to $\ep>0$, suppose $c(a, \mathbb V_h)$ is obtained via a normalized ${\rm osc}(h)$-gapped restriction $\mathbb V_h(\mathfrak a)$ for some index sequence $\mathfrak a: \Z \to \R \backslash \mathcal S_h$. Now, consider a new sequence $\mathfrak b: \Z \to \R\backslash \mathcal S_h$ (note that $\mathcal S_h = \mathcal S_{\bar{h}}$) defined by 
\[ \mathfrak b(i)  = - \mathfrak a(i) = - \eta_i. \]
Clearly, the associated persistence module $\mathbb V_{\bar{h}}(\mathfrak b)$ is a normalized ${\rm osc}(h)$-gapped restriction of $\mathbb V_{\bar{h}}$. Next, for each $i \in \Z$ and any point $(t,x) \in [0,1] \times M$, we have 
\begin{align*}
\left((-\eta_i) \# \bar{h}\right)(t,x) & = (-\eta_i) + \bar{h}(t, \phi_R^{\eta_i t}(x)) \\
& = -\eta_i - h(-t, \phi_R^{\eta_i t}(x)) \\
& = - \left(\eta_i + h\circ \phi_R^{-\eta_i t}\right)(-t,x) = \left(\overline{\eta_i \#h}\right) (t,x).
\end{align*}
This implies that 
\begin{equation}\label{hf-iso}
{\rm HF}_*(\eta_i \#h)\simeq {\rm HF}^{-*}(\overline{\eta_i \#h}) \simeq {\rm HF}^{-*}((-\eta_i) \# \bar{h}). 
\end{equation}
The first isomorphism in (\ref{hf-iso}) is precisely given by the Poincar\'e duality as in (\ref{chain-iso}). Together with the second isomorphism in (\ref{hf-iso}), we know that $\mathbb V_h(\mathfrak a)$ and $\mathbb V_{\bar{h}}(\mathfrak b)$ are dual to each other. Moreover, if the class $a \in {\rm SH}_*(W)$ is generated by the basis element $\{e_1, ..., e_m\}$, then ${\rm PD}(a)$ is generated by $\{{\rm PD}(e_1), ..., {\rm PD}(e_m)\}$. Since ${\rm PD}$ changes the sign of the filtration, one gets $c({\rm PD}(a), \mathbb V_{\bar{h}}(\mathfrak b)) = -c(a, \mathbb V_h(\mathfrak a))$. Therefore, 
\[ c({\rm PD}(a), \mathbb V_{\bar{h}}) \leq c({\rm PD}(a), \mathbb V_{\bar{h}}(\mathfrak b)) = -c(a, \mathbb V_h(\mathfrak a)) = - c(a, \mathbb V_h). \]
Since $a = {\rm PD}({\rm PD}(a))$, a symmetric argument yields the desired conclusion. \end{proof}

Next, we will prove Theorem \ref{prop:triangle}, which is based on the existence of the product ${\rm HF}_\ast(h)\otimes {\rm HF}_\ast(g)\to {\rm HF}_\ast(h\hat{\#}g)$ from Section~\ref{sec-pp}. This product is, however, not well defined for arbitrary contact Hamiltonians $h,g$ as in Theorem~\ref{prop:triangle}. Namely, for the product to be well defined, the contact Hamiltonians $h_t$ and $g_t$ are required to be equal to 0 for $t$ in a neighbourhood of $\mathbb{Z}\subset \R$. Now, we describe a procedure, the reparametrization of the time interval, that makes arbitrary contact Hamiltonains suitable for the product.

 Let $\mu:[0,1]\to[0,1]$ be a smooth function that is equal to 0 near 0 and to 1 near 1. The function $\mu$ can be extended to a smooth function $\R\to\R$ by requiring it to satisfy $\mu(t+1)= 1+\mu(t)$ for all $t\in R$. In particular, the derivative $\mu':\R\to\R$ is a 1-periodic smooth function that is eqiual to $0$ in a neighbourhood of $\mathbb{Z}\subset\R.$ Given a contact Hamiltonian $f$, denote by $\tilde{f}$ the reparametrization of $f$ via $\mu$, i.e.  $\tilde{f}_t(x):= \mu'(t)\cdot f_{\mu(t)}(x).$ If the contact Hamiltonian $f$ is admissible, then so is $\tilde{f}$ and there exists an isomorphism ${\rm HF}_\ast(f)\to {\rm HF}_\ast(\tilde{f})$ given on generators by $\gamma\mapsto \gamma\circ\mu$. This isomorphism commutes with the continuation maps. In particular, the persistent modules $\{{\rm HF}_\ast(\eta\#h)\}_{\eta \in \R \backslash \mathcal S_h}$ and $\{{\rm HF}_\ast(\widetilde{\eta\#h})\}_{\eta \in \R \backslash \mathcal S_h}$ are isomorphic.

In fact, even a ``partial'' reparametrization gives rise to an isomorphism of Floer homologies. Namely, there exists an isomorphism ${\rm HF}_\ast(h\#g)\to {\rm HF}_\ast(\tilde{h}\#g)$, where $h$ and $g$ are contact Hamiltonians such that $h\:\#\:g$ is admissible. This isomorphism is defined as follows. Let $H_t, G_t:\widehat{W}\to\R$ be non-degenerate Hamiltonians that have slopes $h$ and $g$, respectively. Denote by $\tilde{H}_t:\widehat{W}\to\R$ the Hamiltonian given by $\tilde{H}_t(p):= \mu'(t)\cdot H_{\mu(t)}(p)$. The time-1 maps of the Hamiltonians $H$ and $\tilde{H}$ coincide. Therefore, $\phi_{\tilde{H}\# G}^1= \phi_{{H\# G}}^1. $ Denote by $\mathcal{F}$ the set of fixed points of the map $\phi_{\tilde{H}\# G}^1= \phi_{{H\# G}}^1.$ The generators of ${\rm HF}_\ast(H\#G)$ are of the form $t\mapsto\phi^t_{H\#G}(p)$, where $p\in\mathcal{F}$. Similarly, the generators for ${\rm HF}_\ast(\tilde{H}\#G)$ are of the form $t\mapsto \phi^t_{\tilde{H}\# G}(p)$, where $p\in\mathcal{F}$. The isomorphism ${\rm HF}_\ast(H\#G)\to {\rm HF}_\ast(\tilde{H}\# G)$ on generators is given by
\[\phi^\bullet_{H\#G}(p)\mapsto \phi^\bullet_{\tilde{H}\# G}(p) \]
for any $p\in\mathcal{F}$. 
This isomorphism commutes with the continuation maps and is an isomorphism already on the chain level. As a consequence, the Floer homology groups ${\rm HF}_\ast(f\# h)$ and ${\rm HF}_\ast(k\# h)$ are isomorphisms if $f_t:\partial W\to\R$ is a contact Hamiltonian that depends only on time and $k:=\int_0^1 f_t dt\in \R$. This fact will be used in the proof of Theorem~\ref{prop:triangle}. 

The reparametrization isomorphisms do not change the image under canonical morphisms, in other words, the following diagram (consisting of a reparametrization isomorphism and canonical morphisms) commutes
\begin{equation}\label{cd:reparcan}
    \begin{tikzcd}
        &{\rm SH}_\ast(W)& \\
        {\rm HF}_\ast(h\# g)\arrow{rr}\arrow{ru} & & {\rm HF}_\ast(\tilde{h}\# g)\arrow{lu} 
    \end{tikzcd}
\end{equation}
This is a consequence of Section~9 in \cite{U-selective}. Indeed, the reparametrization isomorphism
\[{\rm HF}_\ast(h\# g)\to {\rm HF}_\ast(\tilde{h}\# g)\]
can be seen as the isomorphism
\[\mathcal{B}(\{f^s\}): {\rm HF}_\ast(f^0= h\# g)\to {\rm HF}_\ast(f^1= \tilde{h}\# g)\]
that is associated to a smooth family $f_t^s:\partial W\to\R$ of admissible contact Hamiltonians. The family $f^s$ can be obtained by smoothly changing the parametrization of $h$ from the identity $t\mapsto t$ to $\mu$. For more details, see Section \ref{sec-zigzag}. If $\overline{f}^s_t:\partial W\to\R$ is some other smooth family of admissible contact Hamiltonians such that $\overline{f}^s_t\geqslant f_t^s$ for all $s$ and $t$, then the diagram 
\[
\begin{tikzcd}
    {\rm HF}_\ast(f^0)\arrow{r}{\mathcal{B}(\{f^s\})}\arrow{d}& {\rm HF}_\ast(f^1)\arrow{d}\\
    {\rm HF}_\ast(\overline{f}^0) \arrow{r}{\mathcal{B}(\{\overline{f}^s\})}& {\rm HF}_\ast(\overline{f}^1)
\end{tikzcd}
\]
commutes (the vertical arrows correspond to the continuation maps). By taking $ \overline{f}^s_t=a$ for $a\in\R$ such that $a\geqslant f^s_t$, we obtain that the diagram
\[
\begin{tikzcd}
        &{\rm HF}_\ast(a)& \\
        {\rm HF}_\ast(f^0)\arrow{rr}\arrow{ru} & & {\rm HF}_\ast(f^1)\arrow{lu} 
    \end{tikzcd}
\]
commutes. This directly implies that the diagram \eqref{cd:reparcan} commutes as well.

\begin{proof}[Proof of Theorem~\ref{prop:triangle}]
    Let $\mu:\R\to\R$ be a non-decreasing smooth function such that $\mu(t+1)= 1+ \mu(t)$ for all $t\in\R$ and such that $\mu$ is equal to 0 near 0 and to 1 near 1. For a contact Hamiltonian $f_t:\partial W\to\R$ denote by $\tilde{t}_t:\partial W\to\R$ the contact Hamiltonian given by $\tilde{f}_t(x):=\mu'(t)\cdot f_{\mu(t)}(x).$ Let $a, b\in R$. The following inequality holds:
    \begin{align*}
        \left(\widetilde{a\#h}\right)_t =& \left(\tilde{a}\#\tilde{h}\right)_t\\
        =& \mu'(t)\cdot a + \mu'(t)\cdot h_{\mu(t)}\circ\varphi_R^{-a\mu(t)}\\
        =& \mu'(t)\cdot a + \mu'(t)\cdot h_{\mu(t)} +\mu'(t)\cdot \left(h_{\mu(t)}\circ\varphi_R^{-a\mu(t)} - h_{\mu(t)}\right) \\
        \leqslant &  \mu'(t)\cdot a + \mu'(t)\cdot h_{\mu(t)} + \overline{\op{osc}}\:h.
    \end{align*}
    Similarly, we have $(\widetilde{b\#g})_t \leqslant \mu'(t)\cdot b + \mu'(t)\cdot g_{\mu(t)} + \overline{\op{osc}}\:g$. 
    Denote by $\mu_L, \mu_R:\R\to\R$ the smooth functions with 1-periodic derivatives given by $\mu_L(0)=\mu_R(0)=0$ and
    \begin{align*}
        & \mu_L'(t):=\left\{\begin{matrix} \mu'(2t) & t\in\left[0, \frac{1}{2}\right]\\ 0 & t\in\left[\frac{1}{2}, 1\right]\end{matrix}\right.\\
        & \mu_R'(t):=\left\{\begin{matrix} 0 & t\in\left[0, \frac{1}{2}\right]\\ \mu'(2t-1) & t\in\left[\frac{1}{2}, 1\right]. \end{matrix}\right.
    \end{align*}
    In particular, we have $\mu'_L(t)+\mu'_R(t)= \mu'(2t)$. The inequalities above imply
    \begin{align*}
        \left(\widetilde{a\# h}\right)\hat{\#}\left(\widetilde{b\#g}\right) \leqslant & \:\mu_L'(t) a + \mu_L'(t)h_{\mu_L(t)} + \mu_L'(t) \:\overline{\op{osc}} h\\
        & + \mu_R'(t) b + \mu_R'(t)h_{\mu_R(t)} + \mu_R'(t) \:\overline{\op{osc}} g\\
        = &\:\mu_L'(t)\left(a + \overline{\op{osc}} h\right) + \mu_R'(t)\left(b+ \overline{\op{osc}} g\right) + \tilde{h}\hat{\#}\tilde{g}.
    \end{align*}
    Denote $k:= (\max_t \mu'(t)) \cdot \max\{ \overline{\op{osc}}\: h, \overline{\op{osc}}\: g \}$ and denote \[\nu(t):=\int_0^t \left(\mu_L'(t)\left(a + \overline{\op{osc}} h\right) + \mu_R'(t)\left(b+ \overline{\op{osc}} g\right) + k\right) ds.\]
    Then,
    \begin{align*}
        \left(\widetilde{a\# h}\right)\hat{\#}\left(\widetilde{b\#g}\right) \leqslant & \nu'(t) - k + \tilde{h}\hat{\#}\tilde{g}\\
         = &\nu'(t)-k + \left(\tilde{h}\:\hat{\#}\:\tilde{g}\right)\circ\varphi_R^{-\nu(t)} + \left( \tilde{h}\hat{\#}\tilde{g} -\left(\tilde{h}\hat{\#}\tilde{g}\right)\circ\varphi_R^{-\nu(t)}\right)\\
        \leqslant & \nu'(t) - k + \left(\tilde{h}\hat{\#}\tilde{g}\right)\circ\varphi_R^{-\nu(t)} + \overline{\op{osc}}\left(\tilde{h}\hat{\#}\tilde{g}\right)\\
        \leqslant & \nu'(t)+ \left(\tilde{h}\hat{\#}\tilde{g}\right)\circ\varphi_R^{-\nu(t)}\\
        = & \nu\#\left(\tilde{h}\hat{\#}\tilde{g}\right).
    \end{align*}
    In the inequality above, we used $\overline{\op{osc}}(\tilde{h}\hat{\#}\tilde{g})\leqslant k.$ Therefore, by Section \ref{sec-pp}, there is a well-defined product map
    \[{\rm HF}_\ast(\widetilde{a\# h})\otimes {\rm HF}_\ast(\widetilde{b\#g})\to {\rm HF}_\ast\left(\nu\# \left(\tilde{h}\hat{\#}\tilde{g}\right)\right)\]
    whenever the contact Hamiltonians in question are admissible. By reparametrization and partial reparametrization, the Floer homologies ${\rm HF}_\ast(\widetilde{a\# h})$, ${\rm HF}_\ast(\widetilde{b\#g})$, and ${\rm HF}_\ast(\nu\:\# (\tilde{h}\:\hat{\#}\:\tilde{g}))$ are isomorphic to ${\rm HF}_\ast(a\#h)$, ${\rm HF}_\ast(b\#g)$, and 
    \[ {\rm HF}_\ast\left(\left( a +\overline{\op{osc}}\: h + b + \overline{\op{osc}}\:g + k\right)\#(h\hat{\#}g)\right),\]
    respectively. In particular, if
    \begin{align*}
        & \theta\in\op{Im}\left(\iota_{a}: {\rm HF}_\ast(a\# h)\to {\rm SH}_\ast(W)\right),\\
        & \eta\in\op{Im}\left(\iota_b: {\rm HF}_\ast(b\# g)\to {\rm SH}_\ast(W)\right),
    \end{align*}
    then
    $\theta\ast \eta \in\op{im}(\iota_{a+b+\Delta}: {\rm HF}_\ast((a+b+\Delta)\#(h\hat{\#}g))\to {\rm SH}_\ast(W)).$
    Here, 
    $\Delta:= \overline{\op{osc}}\: h + \overline{\op{osc}}\:g + (\max_t \mu'(t)) \cdot \max\{ \overline{\op{osc}}\: h, \overline{\op{osc}}\: g \}.$ Therefore, by (\ref{dfn-si-classical}), 
    \[c(\theta\ast\eta, \mathbb V_{h\hat{\#}g})\leqslant c(\theta, \mathbb V_h) + c(\eta, \mathbb V_g) + \Delta.\]
    Since $\overline{\op{osc}}\: h, \overline{\op{osc}}\:g \leqslant (\max_t \mu'(t)) \cdot \max\{ \overline{\op{osc}}\: h, \overline{\op{osc}}\: g \}$, and since $\mu$ can be chosen such that $\mu'(t)\leqslant 1+\varepsilon$ for arbitrary $\varepsilon>0$, we obtain the desired inequality 
\end{proof}

\section{Descended definitions} \label{sec-more-wd}

In this section, we prove Theorem \ref{thm-descend}, which will be divided into two situations. 

\subsection{Zigzag isomorphism} \label{sec-zigzag} Recall that for any $s$-smooth family of {\it admissible} contact Hamiltonians $h^s: [0,1] \times M \to \R$ for $s \in [0,1]$, in \cite{U-selective}, one can construct an isomorphism 
\[\mathcal B(\{h^s\}_{s \in [0,1]}): {\rm HF}_*(h^0) \to {\rm HF}_*(h^1). \]
More explicitly, different from the isomorphism between ${\rm HF}_*(h^0)$ and ${\rm HF}_*(h^1)$ via the bifurcation method carried out in \cite[Theorem 1]{UZ22}, the isomorphism $\mathcal B(\{h^s\}_{s \in [0,1]})$ is constructed by the following zigzag approach, 
\[ \xymatrixcolsep{1.5pc}\xymatrix{
{\rm HF}_*(h^0)  && {\rm HF}_*(h^{s_1})  && \,\,\,\,\,\cdots &&  {\rm HF}_*(h^1)\\
& {\rm HF}_*(g^0) \ar[lu]_-{\Phi^0} \ar[ru]^-{\Psi^1}  && \cdots \ar[lu]_-{\Phi^1} \,\,\, \ar[ru] && {\rm HF}_*(g^{k-1}) \ar[lu] \ar[ru]^-{\Psi^{k-1}}}
\]
that is, $\mathcal B(\{h^s\}_{s \in [0,1]}) : = \Psi^{k-1} \circ \cdots \circ (\Phi^1)^{-1} \circ \Psi^1 \circ (\Phi^0)^{-1}$. Here, $h^{s_i}$ for $i = 0, ..., k$ are admissible contact Hamiltonians that are discretely chosen along the isotopy $h^s$ where $h^{s_0} = h^0$ and $h^{s_k} = h^1$; the admissible contact Hamiltonian $g^{j}$ for $j = 0, ..., k-1$ are chosen such that $g_j \leq \min\{h^{s_j}, h^{s_{j+1}}\}$; all the $\Phi^i$ and $\Psi^j$ are isomorphisms given by the Floer continuation maps. Note that one can make such choices due to the important observation: the set of admissible contact Hamiltonians is an open subset of all the smooth functions on $[0,1] \times M$ (hence, in particular, each $g_j$ are sufficiently close to $h^{s_j}$ and $h^{s_{j+1}}$, say, in a $C^2$-sense). Moreover, we emphasize that the isomorphism $\mathcal B(\{h^s\}_{s \in [0,1]})$ preserves the gradings. Finally, one can easily check the induced isomorphism $\mathcal B$ above satisfies the following composition-concatenation property, 
\begin{equation} \label{cc-formula}
\mathcal B(\{f^s\}_{s \in [0,1]} \circ \{h^s\}_{s \in [0,1]}) = \mathcal B(\{f^s\}_{s \in [0,1]}) \circ \mathcal B(\{h^s\}_{s \in [0,1]})
\end{equation} 
whenever $f^1 = h^0$. Here, the notation $\circ$ for two $s$-smooth families is the concatenation of two paths and it is parametrized by $s \in [0,2]$. 

Now, for our purpose, consider a contractible loop $\phi = \{\phi^t\}_{t \in [0,1]}$ in ${\rm Cont}(M, \xi)$. Pick a disk-family of contactomorphisms $\sigma: {\mathbb D}^2 \to {\rm Cont}(M, \xi)$ denoted by $\sigma^s_t$, which provides a homotopy from $\mathds{1}$ to $\phi$. In particular, for each intermediate homotopy parameter $s \in [0,1]$, the contact Hamiltonian of contact isotopy $\{\sigma^s_t\}_{t \in [0,1]}$ is denoted by $f^s_t: [0,1] \times M \to \R$. Then, for any admissible contact Hamiltonian $h_t: [0,1] \times M \to \R$, consider the following $s$-smooth family of admissible contact Hamiltonians 
\[ h^s_t: = {f^s_t} \# h_t \,\,\,\,\mbox{for $s \in [0,1]$}. \]
Then, induced by this $h^s_t$, the zigzag construction above implies that we have an isomorphism,  
\begin{equation} \label{B-sigma}
\mathcal B(\sigma) = \mathcal B(\{h^s\}_{s \in [0,1]}): {\rm HF}_*(h) \to {\rm HF}_*(f \# h).
\end{equation}
This isomorphism formally depends on the choice of $\sigma$, where the difference of two ``cappings'' $\sigma$ and $\sigma'$ represents an element in the group $\pi_2({\rm Cont}(M, \xi))$. However, the discussion so far only relies on the contact manifold $(M, \xi)$ without too much referring to the Liouville filling $(W, \alpha)$ (cf. the identification (\ref{identification})). 

Next, let us consider the following set 
\begin{equation} \label{wt-hf}
\widetilde{{\rm HF}}_*(h) : = \frac{{\rm HF}_*(h)}{\pi_2({\rm Cont}(M, \xi))}.
\end{equation}
Here is the key proposition that directly implies the conclusion of Theorem \ref{thm-descend} under the triviality hypothesis on $\pi_2({\rm Cont}(M, \xi))$. 

\begin{prop} \label{prop-hf-wt} With the definition in (\ref{wt-hf}), for any contact Hamiltonians $h, f: [0,1] \times M \to \R$ such that $\{\phi^t_f\}_{t \in [0,1]}$ is a contractible loop in ${\rm Cont}(M, \xi)$, then $\widetilde{{\rm HF}}_*(h) \simeq \widetilde{{\rm HF}}_*(f \# h)$. In particular, $\widetilde{{\rm HF}}_*$ is well-defined on $\widetilde{{\rm Cont}}_0(M, \xi)$. \end{prop}

\begin{proof} Consider two disk-families of contactomorphisms $\sigma, \tau: {\mathbb D}^2 \to {\rm Cont}(M, \xi)$ that bound the same contact isotopy at time-1, that is, $\sigma^1_t = \tau^1_t$. Denote the corresponding contact Hamiltonian these time-1 maps by $f$. The following diagram 
\[ \xymatrixcolsep{5pc} \xymatrix{
{\rm HF}_*(h) \ar[r]^{\mathcal B(\sigma)} \ar[d]_{\mathds{1}} & {\rm HF}_*(f \#h) \ar[d]^-{\mathds{1}}\\
{\rm HF}_*(h) \ar[r]^{\mathcal B(\tau)} & {\rm HF}_*(f \#h)} \]
is not necessarily commutative. However, the difference is given by the composition, 
\[ \mathcal B(\sigma)  \circ \mathcal B(\tau)^{-1} = \mathcal B(\sigma \circ \tau^{-1}) \]
where the equality comes from the formula (\ref{cc-formula}). Here, $\tau^{-1} = \tau^{1-s}$ (hence, $\sigma^1 = (\tau^{-1})^0$ by our assumption). Moreover, the concatenation $\sigma \circ \tau^{-1}$ represents an $S^2$-family of contactomorphisms on $(M, \xi)$. Furthermore, one can verify that the induced morphism $\mathcal B$ only replies on the homotopy class of $\sigma \circ \tau^{-1}$. Therefore, up to the actions by elements in $\pi_2({\rm Cont}(M, \xi))$, the diagram above commutes. In this way, we obtain the desired conclusion. 
\end{proof}

\subsection{Naturality} Recall that $\pi_1({\rm Ham}_c(\widehat{W}, \omega))$ denotes the group of compactly supported Hamiltonian diffeomorphisms on the completion $\widehat{W}$. By definition, each element in $\pi_1({\rm Ham}_c(\widehat{W}, \omega))$, say $\phi = \{\phi_F^t\}_{t \in [0,1]}$, acts as a natural identification on ${\rm HF}_*(h)$ whenever an extended Hamiltonian of $h$, denoted by $H$, on $[0,1] \times \widehat{W}$ is fixed. Namely, suppose ${\rm HF}_*(h) = {\rm HF}_*(H)$, then, on the chain level, one can consider the action defined by 
\begin{equation} \label{identification}
\mathcal N(F): \gamma(t) \longrightarrow (\phi_F^t)^{-1} (\gamma(t))
\end{equation}
for any generator $\gamma(t)$ of the Floer chain complex ${\rm CF}_*(H)$ of ${\rm HF}_*(H)$. This was considered in \cite[Section 2.7]{U-auto}. This provides a straightforward isomorphism between ${\rm HF}_*(H)$ and ${\rm HF}_*(\bar{F} \#H)$, where $\bar{F} \# H$ generates the Hamiltonian isotopy $(\phi_F^{t})^{-1} \circ \phi_H^t$ with an explicit formula given as follows, 
\begin{equation} \label{cocycle-s}
(\bar{F} \# H)_t = (H_t - F_t) \circ \phi_F^t.
\end{equation}
On the one hand, observe that in the convex end, since $H_t(r,x) = e^r \cdot h_t(x)$ for $r \in \R$ and similarly to $F_t$, in the convex end, we have 
\begin{align*}
(\bar{F} \# H)_t(r,x) & = (H_t - F_t)(r - \kappa(f_t)(x), \phi_f^t(x))\\
& = e^r \cdot \left(e^{-\kappa(\phi_f^t)(x)} \cdot \left((h_t - f_t) \circ \phi_f^t\right)(x)\right) = e^r \cdot (\bar{f} \# h)_t(x). 
\end{align*}
By (\ref{cocycle}), we conclude that ${\rm HF}_*(\bar{F} \#H)$ defines ${\rm HF}_*(\bar{f} \# h)$. On the other hand, for any other choice of the extended Hamiltonian $G$ of $h$, by \cite[Lemma 2.29]{U-auto}, we have the following commutative diagram, 
\begin{equation} \label{N-C-com}
\xymatrixcolsep{5pc} \xymatrix{
{\rm HF}_*(H) \ar[r]^-{\mathcal N(F)} \ar[d]_-{\Phi_{H, G}} & {\rm HF}_*(\bar{F} \#H) \ar[d]^-{\Phi_{\bar{F} \#H, \bar{F} \#G}} \\
{\rm HF}_*(G) \ar[r]^-{\mathcal N(F)} & {\rm HF}_*(\bar{F} \#G)}
\end{equation}
where $\Phi_{H, G}$ and $\Phi_{\bar{F} \#H, \bar{F} \#G}$ are Floer continuation maps. This implies that the action in (\ref{identification}) by $\pi_1({\rm Ham}_c(\widehat{W}, \omega))$ is in fact well-defined on ${\rm HF}_*(h)$. 

Similarly to (\ref{wt-hf}), this leads us to consider the following set
\begin{equation} \label{dfn-red-cHF}
\overline{{\rm HF}}(h) := \frac{{\rm HF}_*(h)}{\pi_1({\rm Ham}_c(\widehat{W}, \omega))}
\end{equation}
a reduced version of the contact Hamiltonian Floer homology, obtained by modulo the action by $\pi_1({\rm Ham}_c(\widehat{W}, \omega))$ as in (\ref{identification}). Here, let us emphasize that the action in (\ref{identification}) usually does {\it not} preserve the grading, therefore, in definition (\ref{dfn-red-cHF}), we drop the grading on purpose. 

Here is the key proposition that implies the conclusion of Theorem \ref{thm-descend} under the triviality hypothesis on $\pi_1({\rm Ham}_c(\widehat{W}, \omega))$. 

\begin{prop} \label{prop-hf-bar} With the definition (\ref{dfn-red-cHF}), for any contact Hamiltonian $ h, f: [0,1] \times M \to \R$ such that $\{\phi^t_f\}_{t \in [0,1]}$ is a contractible loop in ${\rm Cont}(M, \xi)$, then $\overline{{\rm HF}}(h) \simeq \overline{{\rm HF}}(\bar{f} \# h)$. In particular, $\overline{{\rm HF}}$ is well-defined on $\widetilde{{\rm Cont}}_0(M, \xi)$. \end{prop}

\begin{proof} By the definition in (\ref{dfn-red-cHF}), we need to show that for any extended Hamiltonian $H$ of $h$ defined on $[0,1] \times \widehat{W}$, there exists an extended Hamiltonian $K$ of $\bar{f} \#h$ defined on $[0,1] \times \widehat{W}$ such that the corresponding Hamiltonian Floer homologies ${\rm HF}_*(H)$ and ${\rm HF}_*(K)$ are isomorphic (not necessarily preserving the grading), and this isomorphism is independent of the choice of the extended Hamiltonian, up to the action by $\pi_1({\rm Ham}_c(\widehat{W}, \omega))$. 

By the Biran-Giroux long exact sequence (see \cite{BG-long} or \cite[Section 2.3]{DU-exotic}), there exists $F: [0,1] \times \widehat{W} \to \R$, a Hamiltonian with slope $f$, such that $\{\phi_F^t\}_{t \in [0,1]}$ is a loop in ${\rm Ham}(\widehat{W}, \omega)$ (not necessarily contractible). Then the Hamiltonian $K_F: = \bar{F} \# H$, defined in (\ref{cocycle-s}), generates the Hamiltonian isotopy $(\phi_F^{t})^{-1} \circ \phi_H^t$. The identification in (\ref{identification}) provides an isomorphism $\mathcal N(F): {\rm HF}_*(H) \simeq {\rm HF}_*(K_F)$. 

However, this is not enough to yield the same conclusion for contact Hamiltonian Floer homologies ${\rm HF}_*(h)$ and ${\rm HF}_*(\bar{f} \#h)$, since the identification in $\mathcal N(F)$ depends on the choice of the extended Hamiltonian $F$. Explicitly, suppose $G: [0,1] \times \widehat{W} \to \R$ is another extended Hamiltonian that has slope $f$ and $\{\phi_G^t\}_{t \in [0,1]}$ is a loop in ${\rm Ham}(\widehat{W}, \omega)$. Similarly to $K_F$, denote the resulting Hamiltonian by $K_G: = \bar{G} \#H$. Different from (\ref{N-C-com}), the following diagram does not always commute, 
\[ \xymatrixcolsep{5pc} \xymatrix{
{\rm HF}_*(H) \ar[r]^{\mathcal N(F)} \ar[d]_{\mathds{1} = \Phi_{H,H}} & {\rm HF}_*(K_F) \ar[d]^-{\Phi_{K_F, K_G}}\\
{\rm HF}_*(H) \ar[r]^{\mathcal N(G)} & {\rm HF}_*(K_G)} \]
But, the following isomorphism
\[ \mathcal N(\bar{F} \# G) = \mathcal N(F)^{-1} \circ \mathcal N(G): {\rm HF}_*(K_F) \to {\rm HF}_*(K_G) \]
is induced by the action of the element $\phi_{\bar{F} \#G}^1 \in \pi_1({\rm Ham}_c(\widehat{W}, \omega))$. Indeed, since by the formula (\ref{cocycle-s}), the Hamiltonian $\bar{F} \# G$ has slope $0$ in the convex end where the Hamiltonian isotopy $\phi_{\bar{F} \# G}^t = \mathds{1}$. In other words, both ${\rm HF}_*(K_F)$ and ${\rm HF}_*(K_G)$ represent the same set $\overline{{\rm HF}}(\bar{f} \# h)$. So, we can take the requested $K$ to be $K_F$ for any extended Hamiltonian $F$ of $f$. Moreover, the isomorphism $\mathcal N(F)$ from (\ref{identification}) is the desired isomorphism on the level of $\overline{{\rm HF}}$. This completes the proof. 
\end{proof}

\medskip

\noindent {\bf Acknowledgement}. The second author is supported by the Science Fund of the Republic of Serbia, grant no.~7749891, Graphical Languages - GWORDS. The third author is supported by USTC Research Funds of the Double First-Class Initiative. Part of this paper has been presented in the Topology Seminar at the Yau Mathematical Sciences Center, Tsinghua University in May 2023, as well as in the conference Persistence Homology in Symplectic and Contact Topology, held in Albi, France in June 2023. We thank Honghao Gao and Jean Gutt for their invitations and hospitality. We also thank Jungsoo Kang for his interest in our work and many fruitful communications. Finally, we express special gratitude to Dylan Cant for thorough and generous communications on several overlaps between our works. In addition, we thank Professor Dietmar Salamon for a general discussion on the draft of this paper.

\bibliographystyle{amsplain}
\bibliography{biblio_CHD}

\noindent\\

\end{document}